\documentclass[12pt]{elsarticle}
\usepackage{geometry}
\usepackage{lineno,hyperref}
\usepackage{cancel}
\usepackage{graphicx}
\modulolinenumbers[5]
\usepackage{xcolor}
\usepackage{bm}
\usepackage{amsmath}
\usepackage{soul}

\usepackage{amsthm}
\usepackage{newtxmath}
\theoremstyle{theorem} 
   \newtheorem{theorem}{Theorem}[section]
   \newtheorem{corollary}[theorem]{Corollary}
   
   \newtheorem{proposition}[theorem]{Proposition}
\theoremstyle{definition}
   \newtheorem{definition}[theorem]{Definition}
   \newtheorem{example}[theorem]{Example}
\theoremstyle{remark}

\begin{document}

\begin{frontmatter}

\title{Projective Embedding of Dynamical Systems:\\\textit{ uniform mean field equations}}

\author{F. Caravelli}

\address{Los Alamos National Laboratory\\
{T-Division (T-4), Condensed Matter \& Complex Systems,\\ Los Alamos, New Mexico 87545, USA}}
\ead{caravelli@lanl.gov}
\author{F. L. Traversa}
\address{MemComputing Inc.\\
{MemComputing Inc., 9909 Huennekens Street, Suite 110,
San Diego, California 92121, USA}
}

\author{M. Bonnin}
\author{F. Bonani}
\address{Politecnico di Torino\\
{Department of Electronics and Telecommunication,  Corso Duca degli
Abruzzi 24, 10129 Turin, Italy}}


\begin{abstract}
We study embeddings of continuous dynamical systems in larger dimensions via projector operators. We call this technique PEDS, projective embedding of dynamical systems, as the stable fixed point of the original system dynamics are recovered via projection from the higher dimensional space. In this paper we provide a general definition and prove that for a particular type of rank-1 projector operator, the uniform mean field projector, the equations of motion become a mean field approximation of the dynamical system. While in general the embedding depends on a specified variable ordering, the same is not true for the uniform mean field projector. In addition, we prove that the original stable and saddle-node fixed points retain this feature in the embedding dynamics, while unstable fixed points become saddles. 
Direct applications of PEDS can be non-convex optimization and machine learning. 
\end{abstract}


\end{frontmatter}




\section{Introduction}
The past decades witnessed an increased interest in 
physics- or neuro-inspired algorithms for the analysis of dynamical systems, with the main area of application being problems that can be mapped onto optimization ones, whether continuous or discrete
\cite{ Ventra2018,TraversaSOLG,Kirkpatrick,qannealer,tunnelingbald,Hennessy2019,Vadlamani2020,traversa,Sutton2017,Isingmachine,Pierangeli_2019,Csaba2020,goto,Dorigo2004}. 
Among the most important neuro-inspired algorithms, we mention neural networks, which received a large amount of attention given their wide applicability and remarkable achievements: this is an active area of research falling at the boundary between complex systems, neuromorphic computing and nonlinear dynamics, dating back to Turing \cite{Turing} at least. 
In the study of neural networks, one of the most important open problems is the acceleration of the training phase, a problem connected to the roughness of the energy landscape \cite{mckay, barber}. Network training is one of the most difficult tasks, requiring in general huge computational power and a vast number of samples. Many algorithms attempt at modifying the energy function to reduce the time spent on saddle points \cite{jordan,jordan2}. Changing the landscape is however challenging in general, as it somehow requires some \textit{a priori} knowledge of what type of local extrema  should be modified.
Thus, finding valuable alternatives  and/or  generalizations of gradient descent has been a topic of intense study. In addition to this, analog models of computation is an active area of research \cite{pouly} with several applications.

From the point of view of a dynamical system, however, there are not many strategies that one can attempt to employ. A possibility, incidentally the one we explore in this paper, is to increase the dimensionality of the system, by attempting to preserve some properties related to the original dynamical system, while aiming at a trade-off between convergence optimality and the curse of dimensionality.
The basic rationale for this strategy is that increasing dimensions, there are more  pathways that a system can take in order to reach a certain target point.  As a simple example, consider a one dimensional barrier between two minima in a potential: following gradients, one could never move from one local minimum to the other, while in a higher dimension system, pathways around that confinement barrier are, at least in principle, possible.

The technique we propose here is inspired by recent results in the context of memristive circuits \cite{chua71,stru08,Caravelli2016rl,Caravelli2016ml,Caravelli2017,Caravelli2019}. In circuits, Kirchhoff laws are manifestations of the conservation of physical quantities such as charge or energy. Mathematically, these can be expressed via the introduction of projection operators, i.e. matrices $\boldsymbol{\Omega}$ satisfying the constraint $\boldsymbol{\Omega}^2=\boldsymbol{\Omega}$, and directly connected to circuit topology. For instance, for a resistive circuit made of identical unitary resistances in series with impressed voltage generators, the Ohm's law for the network can be expressed as
\begin{eqnarray}
    \vec i=\boldsymbol{\Omega} \vec v,
    \label{eq:rescirc}
\end{eqnarray}
where $\vec v$ is the collection of voltage generators connected in series to each resistance, while $\vec i$ contains the branch currents. The underlying assumption of \eqref{eq:rescirc} is that the voltage generators $v_i$'s are in series to the resistances $i$'s, while the circuit can be represented as graph with $E$ edges. Given the branch currents and a certain orientation of the graph loops $1,\dots, L$, we can obtain the so called loop matrix of the circuit $A$, of size $L\times E$, such that $\boldsymbol{\Omega}=\boldsymbol{A}^t(\boldsymbol{A} \boldsymbol{A}^t)^{-1} \boldsymbol{A}$, where $^t$ denotes the transpose.
The details of the derivation of $\boldsymbol{\Omega}$ from the circuit topology are beyond the scope of this paper, where $\boldsymbol{\Omega}$ will be kept generic and unrelated to any underlying graph or conservation law.

We assume a continuous dynamical system, but the technique can in principle be extended to vector maps, and thus works also for numerical implementations of a dynamical system. Let us consider a dynamical system expressed in vector form as a first-order differential system
\begin{eqnarray}
    \frac{dx_i}{dt} =  f_i(\vec x) \qquad i=1,\dots,m
    \label{eq:orig}
\end{eqnarray}
where functions $f_i(\cdot)$ are assumed known, and analytic. We are in general interested in recovering the stable fixed points of \eqref{eq:orig}, i.e. the values $\vec x^*$ such that $f_i(\vec x^*)=0$, if they exist.

To this aim we consider another dynamical system, of size $mN$, written in the form
\begin{eqnarray}
    \frac{d}{dt}  \vec X_i=  \boldsymbol{\Omega} \vec F_i(\vec X_1,\cdots, \vec X_m )+\vec G_i(\vec X_i) \qquad i=1,\dots,m
    \label{eq:ext0}
\end{eqnarray}
where for each $i$ value we define an augmented vector  $\vec X_i$ of size $N$. The question we aim at answering in this contribution is to ascertain  whether functions $\vec F_i$ and $\vec G_i$ exist such that the dynamical system \eqref{eq:orig} is contained, in a sense we will make more precise in the next section, into the extended system  \eqref{eq:ext0}. The answer we provide in this paper is affirmative, as we will explicitly construct such system along with the technique to recover the original dynamical system.  

From a mathematical perspective, these generalizations can be investigated by the study of the properties of fixed points in the embedded system in terms of the original ones, which is the strategy we use in this paper. A fixed point $\vec x^*$ is particular point of the phase space satisfying $\frac{d\vec x}{dt}|_{\vec x^*}=f(\vec x^*)=0$.
 We dub the method developed in this paper \textit{Projective Embedding of Dynamical Systems} (PEDS), as the technique involves the embedding of a target dynamical system of dimension $m$ into one of dimension $mN$; ultimately, we recover the fixed points of the original dynamical system by projecting back onto a chosen set (of size $m$)  of observables. We will prove  that the information of the fixed points of the original target system are related to the fixed points of the reduced observables.  As we will see, the dynamical system in which the embedding is contained is a nontrivial and nonlinear extension of the original dynamical system which is obtained via a map between the original one and an extended one. Although the projection operator may be quite general, we prove most of the results here for a specific operator, that we call uniform mean-field projector, as in this simplified case mostly analytical proofs are available.

The structure of the paper is as follows. In Section \ref{sec:defs} we introduce the PEDS procedure formally, and provide various examples to intuitively grasp why these definitions make sense. In Section \ref{sec:umfp} we study the uniform mean field projector, and both for 1-dimensional and $m$-dimensional dynamical systems we prove exact results about the properties of the asymptotic stable fixed points and their Jacobians. In Section \ref{sec:numerics} we provide numerical examples alongside analytical analysis, to further corroborate the bulk of the paper. Finally, conclusions follow.

\section{The PEDS procedure: key definitions and examples} \label{sec:defs}

In order to clarify the techniques developed in this paper, we now construct the simplest example of the embedding, before introducing the necessary definitions. Notation-wise, we will denote with $\boldsymbol{I}$ the identity matrix, while $\vec 1$ is a column vector with elements equal to 1. 

\begin{example} \label{ex:exampleexp}
Exponential dynamics.\\
\rule{\textwidth}{0.05cm}
Let us consider the following one dimensional dynamical system:
\begin{eqnarray}
    \frac{d x}{dt}=a \tilde x \qquad  x(0)=x_0,
\end{eqnarray}
with $a\in \mathbb{R}$, whose analytical solution is given by
\begin{eqnarray}
    x(t)=e^{a t} x_0.
\end{eqnarray}

Considering an $N\times N$ projector matrix $\boldsymbol{\Omega}$ such that $\boldsymbol{\Omega}^2=\boldsymbol{\Omega}$ and, thus, $\boldsymbol{\Omega}(\boldsymbol{I}-\boldsymbol{\Omega})=0$, we define the following enlarged (size $N$) dynamical system
\begin{eqnarray}
    \frac{d\vec X}{dt}=a \boldsymbol{\Omega} \vec X -\alpha (\boldsymbol{I}-\boldsymbol{\Omega} )\vec X \qquad \vec X(0)=x_0 \vec b,
\end{eqnarray}
where $\alpha>0$, and $\vec b$ is an arbitrary vector, satisfying the only requirement $\boldsymbol{\Omega} \vec b\neq \vec 0$.

Since the system above is linear, we do know the analytical solution, which is given by
\begin{eqnarray}
    \vec X(t)=e^{[a \boldsymbol{\Omega} -\alpha (\boldsymbol{I}-\boldsymbol{\Omega} )]t} \vec X(0) \approx e^{a \boldsymbol{\Omega}t} \vec X(0)
    \label{eq:1dex}
\end{eqnarray}
where the approximation holds for $t\rightarrow +\infty$, i.e. for $t\gg 1/\alpha$. As for any projector $\boldsymbol{\Omega}$ the following identity holds
\begin{eqnarray}
    e^{a\boldsymbol{\Omega}}=\boldsymbol{I}+(e^a-1)\boldsymbol{\Omega}
    \label{eq:projident}
\end{eqnarray}
the asymptotic solution reads
\begin{eqnarray}
    \vec X(t) \approx (\boldsymbol{I}-\boldsymbol{\Omega})\vec X(0)+ e^{at}x_0\boldsymbol{\Omega}\vec b = (\boldsymbol{I}-\boldsymbol{\Omega})\vec X(0)+ x(t)\boldsymbol{\Omega}\vec b
    \label{eq:2dex}
\end{eqnarray}
Therefore, projecting \eqref{eq:2dex}
\begin{eqnarray}
    \boldsymbol{\Omega} \vec X(t)\approx x(t) \boldsymbol{\Omega} \vec b,
    \label{eq:1dsol}
\end{eqnarray}
i.e., the asymptotic solution of \eqref{eq:1dex} is contained as a common factor in all the modes of $\vec X(t)$, the ``replicated" dynamics.

As a last comment, we recover the solution of the original dynamical system by averaging the elements of \eqref{eq:1dsol}
\begin{eqnarray}
    \frac{1}{N}\vec 1^T\boldsymbol{\Omega} \vec X(t)\approx x(t) \frac{1}{N} \vec 1^T\boldsymbol{\Omega} \vec b,
\end{eqnarray}
where $^T$ represents the transpose. Therefore, choosing vector $\vec b$ such that
\begin{eqnarray}
   \frac{1}{N} \vec 1^T\boldsymbol{\Omega} \vec b=\frac{1}{N} \sum_{i,j=1}^N \Omega_{ij}b_j=1
\end{eqnarray}
we find
\begin{eqnarray}
    x(t)=\frac{1}{N}\vec 1^T\boldsymbol{\Omega} \vec X(t) = \frac{1}{N} \sum_{i,j=1}^N \Omega_{ij}X_j(t)
\end{eqnarray}
i.e., the projected dynamics recovers the original one dimensional system.\\
\rule{\textwidth}{0.05cm}
\end{example}

The main goal of this paper is to extend the results of Example~\ref{ex:exampleexp} to arbitrary dynamical systems. Let us now identify the key steps of the procedure. First, we begin with a dynamical system in the standard form. 

\begin{figure}
    \centering
    \includegraphics[scale=0.8]{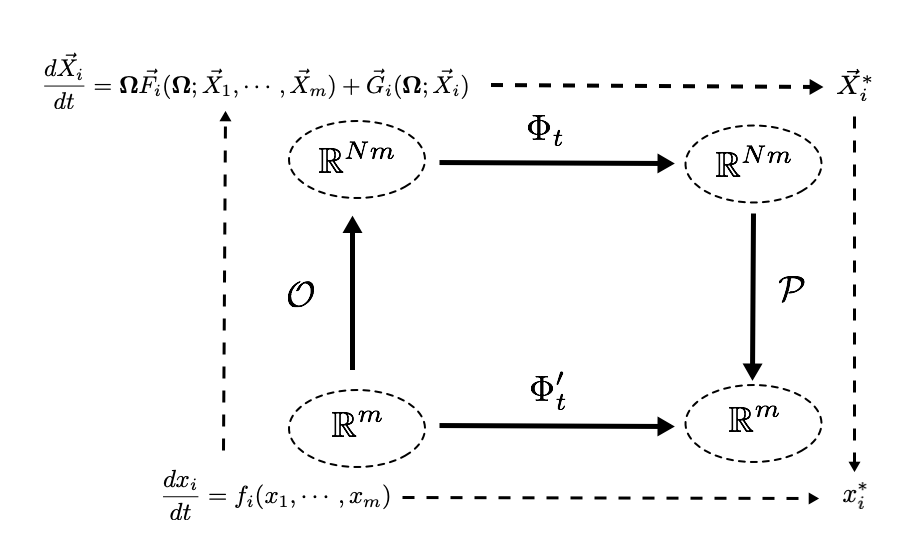}
    \caption{Graphical representation of the PEDS procedure and of the associated maps. The horizontal arrows represent the time evolution maps $\Phi_t$ and $\Phi_t^\prime$, while the vertical arrows represent the embedding $\mathcal O$ and the projection $\mathcal P$ map, respectively.}
    \label{fig:projectedd}
\end{figure}

\begin{definition} \textit{Embedding procedure: PEDS.} \label{def:peds}
We explicitly define here  the steps involved in developing the PEDS procedure.
\begin{enumerate}
\item We begin with a tuple $(\{ f_1(\vec x),\cdots,f_m(\vec x)\},\boldsymbol{\Omega},\{\vec G_1,\cdots, \vec G_m\},\mathcal S,\{\vec b_1,\cdots,\vec b_m\},N)$, where $\boldsymbol{\Omega}$ is a 
size $N$ projector operator. We call  $\{ f_1(\vec x),\cdots,f_m(\vec x)\}$ the \textit{target} dynamical system, while  $x_i$ represent the \textit{target variables}. $\mathcal S$ represents an ordering, relevant for the case of a multi-dimensional target system if the embedding is non-commutative. Vector $\vec b$ is constant and such that $\boldsymbol{\Omega} \vec b\neq \vec 0$. Finally,  $N$ is the dimension of the embedding for each scalar variable. As such, it  can also be interpreted as the number of dimensions in which each scalar variable is expanded into.
\item Given the target dynamical system of dimension $m$, we build an extended dynamical system of size $mN$, represented by a set of canonical equations of the form
\begin{eqnarray}
    \frac{d}{dt}  \vec X_i=  \boldsymbol{\Omega}  \boldsymbol{F}_i(\vec X_1,\cdots, \vec X_m )\vec b_i+\vec G_i(\boldsymbol{\Omega};\vec X_i) \qquad i=1,\dots, m
    \label{eq:ext}
\end{eqnarray}
 This step 
 is represented by the arrow $\mathcal O$ in Fig.~\ref{fig:projectedd}, being  it a mapping between each scalar functions $f_i$ to the vector function $\vec F_i=\boldsymbol{F}_i\vec b_i$. Thus, for each dimension of the original dynamical system, we obtain an extended $N$ dimensional subspace in the $Nm$ dynamical system, so that 
 \begin{eqnarray}
(\{ f_1(\vec x),\cdots,f_m(\vec x)\}) &\underset{\mathcal O}{\rightarrow}& (\{ \vec F_1(\vec X_i),\cdots,\vec F_m(\vec X_i)\},\{\vec G_1(\boldsymbol{\Omega};\vec X_i),\cdots, \vec G_m(\boldsymbol{\Omega};\vec X_i)\})\nonumber \\
&&\ \ \  \equiv (F,G)
  \end{eqnarray}
We call the specific map $\mathcal O$ the \textit{embedding}, while $(F,G)$ is the \textit{extended system} and $\vec X_i$ the \textit{extended variable} (i.e., a set of $N$ scalar variables in the extended system). We also dub the set $\vec G_i$, the \textit{decay functions}. In each extended subspace, $\vec X_i$ is a vector of components $X_{i,j}$, while diagonal matrix $\boldsymbol{X}_i$ is made of elements $\boldsymbol{X}_{i,jk}=X_{i,j}\delta_{jk}$, where $\delta_{jk}$ represents Kronecker symbol. The original vector can be easily recovered from the diagonal matrix as $\vec X_i=\boldsymbol{X}_i \vec 1$. We stress that in principle $\boldsymbol{F}_i$ can be a non-trivial function of $\Omega$, as we shall discuss later on.
\item We  consider the time evolution of both the original and the extended system, represented by maps $\Phi_t^\prime$ and $\Phi_t$, respectively, in Fig.~\ref{fig:projectedd}. In Example \ref{ex:exampleexp}, the two maps were analytically expressed, thanks to the simplicity of the target system. 
\item Arrow $\mathcal P$ in Fig.~\ref{fig:projectedd}, finally, projects the extended dynamical system from size $Nm$ to an $m$ dimensional system, that is required to coincide with the trajectory of the target  dynamical system. For each variable, the projection is derived from the projector operator as, given a certain extended variable $\vec X_i$, we obtain $\bar x_i=\frac{1}{N} \sum_{j,k=1}^N {\Omega}_{jk} X_{i,k}.$
\end{enumerate}
\end{definition}

\subsection{Extended variable ordering}
\label{sec:ordering}
Before delving into the details of the construction of \eqref{eq:ext}, let us clarify what we mean by ordering. During the development of the PEDS procedure, commuting variable products such as $x_1 x_2$ will be mapped onto matrix products of the form $(\boldsymbol{\Omega} \boldsymbol{X}_1)(\boldsymbol{\Omega} \boldsymbol{X}_2)$. 
As matrix products do not commute, the ordering of the variables will have a role.

\begin{definition} \textit{Ordering.} Within the context of PEDS, an ordering $S$ is a map between commuting monomials of the form $x_1^{i_1} x_2^{i_2} \cdots x^{i_m}_m$ and non-commuting matrix monomials of the form $(\boldsymbol{\Omega} \boldsymbol{X}_1)^{i_1}(\boldsymbol{\Omega} \boldsymbol{X}_2)^{i_2}\cdots (\boldsymbol{\Omega} \boldsymbol{X}_m)^{i_m}$.
\end{definition}
In general, an ordering can be written in terms of a certain set of coefficients. We will use the following notation
\begin{eqnarray}
    \{(\boldsymbol{\Omega}  \boldsymbol{X}_1)^{i_1}\cdots (\boldsymbol{\Omega}  \boldsymbol{X}_m)^{i_m} \}_S=\sum_{\sigma\in \mathcal S(m)} o_{\sigma(1) \cdots \sigma(m)} (\boldsymbol{\Omega}  \boldsymbol{X}_{\sigma(1)})^{i_{\sigma(1)}}\cdots (\boldsymbol{\Omega}  \boldsymbol{X}_{\sigma(m)})^{i_{\sigma(m)}}
\end{eqnarray}
where $\sigma$ is an element of the permutation group $\mathcal S(m)$ over $m$ variables, the coefficients $o_{\sigma(1) \cdots \sigma(m)}$ are zero if at least two indices are equal, and
\begin{equation}
    \sum_{\sigma\in \mathcal S(m)} o_{\sigma(1) \cdots \sigma(m)}=1.
\end{equation}

\begin{definition}
Given the monomial $x_1^{i_1} x_2^{i_2} \cdots x_m^{i_m}$, the \textit{standard} ordering is given by $(\boldsymbol{\Omega} \boldsymbol{X}_1)^{i_1} (\boldsymbol{\Omega} \boldsymbol{X}_2)^{i_2}\cdots (\boldsymbol{\Omega} \boldsymbol{X}_m)^{i_m}$, i.e. a matrix monomial where the matrix products strictly follow the same sequence as in the scalar case.
\end{definition}

\begin{definition}
Given the monomial $x_1^{i_1} x_2^{i_2} \cdots x_m^{i_m}$, the \textit{balanced} ordering is given by $$\frac{1}{m!}\sum_{\sigma\in \mathcal S(m)}(\boldsymbol{\Omega} \boldsymbol{X}_{\sigma(1)})^{i_{\sigma(1)}} (\boldsymbol{\Omega} \boldsymbol{X}_{\sigma(2)})^{i_{\sigma(2)}}\cdots (\boldsymbol{\Omega} \boldsymbol{X}_{\sigma(m)})^{i_{\sigma(m)}}$$ 
\end{definition}

Notice that $$\frac{1}{m!}\sum_{\sigma\in \mathcal S(m)} 1=1.$$
and that, given an order-independent function $M$, i.e. a function satisfying
\begin{eqnarray}
    M(a_1,\cdots,a_m)=M(a_{\sigma(1)},\cdots,a_{\sigma(m)}), 
\end{eqnarray}
for any permutation $\sigma\in \mathcal S(m)$, then
\begin{align}
    \sum_{\sigma\in \mathcal S(m)} o_{\sigma(1),\cdots,\sigma(m)}M(a_{\sigma(1)},\cdots,a_{\sigma(m)})&= \sum_{\sigma\in \mathcal S(m) } o_{\sigma(1),\cdots,\sigma(m)} M(a_1,\cdots,a_m) \nonumber \\
    &=M(a_1,\cdots,a_m)\sum_{\sigma\in \mathcal S(m)} o_{\sigma(1),\cdots,\sigma(m)}\nonumber \\ &=M(a_1,\cdots,a_m)
\end{align}

\begin{example}\ \\
\rule{\textwidth}{0.05cm}
The standard ordering is characterized by
\begin{eqnarray}
    o_{\sigma(1) \cdots \sigma(m)}=\delta_{\sigma(1) 1}\cdots \delta_{\sigma(m) m},
\end{eqnarray}
while for the balanced ordering, $o_{\sigma(1) \cdots \sigma(m)}={1}/{m!}$.

Considering the case of two scalar variables (i.e., $m=2$), choosing  $o_{12}=1$ and $o_{21}=o_{11}=o_{22}=0$, we get $\{(\boldsymbol{\Omega}  \boldsymbol{X}_1)^{i_1}(\boldsymbol{\Omega}  \boldsymbol{X}_2)^{i_2} \}_S=(\boldsymbol{\Omega}  \boldsymbol{X}_1)^{i_1}(\boldsymbol{\Omega}  \boldsymbol{X}_2)^{i_2}$. Another possible choice is  $o_{12}=a, o_{21}=b, o_{11}=o_{22}=0$, where $0\le a,b\le 1$ and $a+b=1$, so that 
\begin{equation}
    \{(\boldsymbol{\Omega}  \boldsymbol{X}_1)^{i_1}(\boldsymbol{\Omega}  \boldsymbol{X}_2)^{i_2} \}_S=a (\boldsymbol{\Omega}  \boldsymbol{X}_1)^{i_1}(\boldsymbol{\Omega}  \boldsymbol{X}_2)^{i_2} + b(\boldsymbol{\Omega}  \boldsymbol{X}_2)^{i_2}(\boldsymbol{\Omega}  \boldsymbol{X}_1)^{i_1}
\end{equation}
\rule{\textwidth}{0.05cm}
\end{example}

\subsection{Decay functions}

We discuss now the decay functions $\vec G_i(\boldsymbol{\Omega};\vec X_i)$. The choice made in Example~\ref{ex:exampleexp} was
\begin{equation}
    \vec G(\boldsymbol{\Omega};\vec X)=-\alpha (\boldsymbol{I}-\boldsymbol{\Omega}) \vec X
    \label{eq:decaystandard}
\end{equation}
where $\alpha\ge 0$. This particular choice corresponds to a precise definition:
\begin{definition} \textit{Standard decay function}.
The decay function in \eqref{eq:decaystandard} is called \textit{standard decay function}.
\end{definition}
As seen in Example~\ref{ex:exampleexp}, the standard decay function allowed to recover the target dynamical system dynamics, that in turn was reconstructed in the Span$(\boldsymbol{\Omega})$. The role played by the decay functions is to enforce that in each extended subspace, the modal components associated to the Ker$(\boldsymbol{\Omega})$ are asymptotically vanishing.
\begin{definition}
A decay function $\vec G_i(\boldsymbol{\Omega};\vec{X}_i)$ is \textit{$\boldsymbol{\Omega}$-eligible} if 
\begin{equation}
    \lim_{t \rightarrow +\infty}\boldsymbol{\Omega} \vec G_i(\boldsymbol{\Omega};\vec X_i(t))=\vec 0
\end{equation}
and if the solution of the dynamical system obtained projecting \eqref{eq:ext} onto the Ker$(\boldsymbol{\Omega})$ (i.e., projecting the extended equation via $(\boldsymbol{I}-\boldsymbol{\Omega})$ and defining $\vec X_{ci}(t)=(\boldsymbol{I}-\boldsymbol{\Omega}) \vec X_i(t)$)
\begin{equation}
    \frac{d\vec X_{ci}}{dt}=(\boldsymbol{I}-\boldsymbol{\Omega}) \vec G_i(\boldsymbol{\Omega};\vec X_i)
\end{equation}
is decaying, i.e. if $\lim_{t\rightarrow \infty} \vec X_{ci}(t)=\vec 0$.
\end{definition}
Obviously, the standard decay functions are $\boldsymbol{\Omega}$-eligible.

\subsection{Embedding map $\mathcal O$} 
We are now ready to state the exact definition of the $\mathcal O$ map. However, this step requires to express the nonlinear scalar functions defining the target dynamical system as a power series. From this standpoint, it is convenient to formulate the Taylor expansion of an $m$ variable, scalar analytic function $f(\vec x)$ as a superposition of monomials exploiting Kronecker symbol:
\begin{equation}
f(\vec x)=\sum_{j=0}^\infty \sum_{i_1,\dots, i_j=0}^j\delta_{j,\sum_{k=1}^m i_k} b_{j;i_1\dots i_j} x_1^{i_1}\cdots x_m^{i_m} =\sum_{j=0}^\infty \sum_{i_1,\dots i_j=0}^j a_{j;i_1\dots i_j} x_1^{i_1}\cdots x_m^{i_m},
\label{eq:taylormonomial}
\end{equation}
where $a_{j;i_1\cdots i_m}=\delta_{j,\sum_{k=1}^j i_j} b_{j;i_1\cdots i_j}$.

\begin{definition} \label{def:matrixmap} \textit{Matrix map.}\\
Given a scalar analytic function $f(\vec x)$ with Taylor expansion as in \eqref{eq:taylormonomial}, we call a \textit{matrix map} for $f$ the following construction
\begin{equation}
\boldsymbol{F}(\vec X_1,\dots, \vec X_m)=\sum_{j=0}^\infty \sum_{i_1,\cdots i_j=0}^j a_{j;i_1\dots i_j} \{(\boldsymbol{\Omega_1}\boldsymbol{X}_1)^{i_1} \cdots (\boldsymbol{\Omega_1}\boldsymbol{X}_m)^{i_m}\}_{S}
\label{eq:vectaylorexp}
\end{equation} 
where $S$ is a properly defined ordering.
\end{definition}

\begin{definition} \label{def:omap} 
\textit{Embedding map.}  \\
The embedding map $\mathcal O$ is defined as the tuple $\mathcal O=\Big(\{ f_i\},\boldsymbol{\Omega}, \{\vec G_i\},S,\{\vec b_i\},N \Big)$, where $\vec G_i$ represents decay functions, and $\vec b_i$ is a set of constant vectors satisfying condition $\boldsymbol{\Omega}\vec b_i\neq\vec 0$. The \textit{embedding map} $\mathcal O$ of \eqref{eq:orig} is given by
\begin{equation}
    \frac{dx_i}{dt}=f_i(\vec x) \underset{\mathcal O}{\longrightarrow} \frac{d\vec X_i}{dt}= \boldsymbol{\Omega} \boldsymbol{F}_i(\vec X_i,\cdots,\vec X_m) \vec b_i+\vec G_i(\boldsymbol{\Omega};\vec X_i) \qquad i=1,\dots,m
    \label{eq:embeddingmap}
\end{equation}
where $\boldsymbol{F}_i$ is the matrix map associated to $f_i$ according to Definition~\ref{def:matrixmap}.
\end{definition}

Let us now provide  three examples of matrices $\boldsymbol{F}_i$ which will be used in the following. Each target function is analytical, with series representation as in \eqref{eq:taylormonomial}:
$$f_i(x_1,\cdots,x_m)=\sum_{k=0}^\infty \sum_{j_1,\cdots,j_m}^k a_{i,k;j_1,\cdots,j_m} x_1^{j_1}\cdots x_m^{j_m} 
$$
\begin{definition}
We define the following three possible matrix embeddings:
\begin{itemize}
    \item the \textit{standard commutative map} is
\begin{eqnarray}
    \boldsymbol{F}^{(c)}_i(\boldsymbol{X}_1,\cdots,\boldsymbol{X}_m)=\sum_{k=0}^\infty \sum_{j_1,\cdots,j_m}^k a_{i,k;j_1,\cdots,j_m} \boldsymbol{X}_1^{j_1}\cdots \boldsymbol{X}_m^{j_m},
    \label{eq:standcomm}
\end{eqnarray}
    \item the \textit{mixed commutative map} is
    \begin{eqnarray}
    \boldsymbol{F}^{(mc)}_i(\boldsymbol{\Omega};\boldsymbol{X}_1,\cdots,\boldsymbol{X}_m)=a_{i,0}\boldsymbol{I}+\sum_{k=1}^\infty \sum_{j_1,\cdots,j_m}^k a_{i,k;j_1,\cdots,j_m} (\boldsymbol\Omega(\boldsymbol{X}_1^{j_1}\cdots \boldsymbol{X}_m^{j_m})^{1/k})^k,
    \label{eq:mixedcomm}
\end{eqnarray}
where $a_{i,0}$ denotes the constant term of the series expansion for function $f_i$
\item the \textit{standard non-commutative map} is
\begin{eqnarray}
    \boldsymbol{F}^{(nc)}_i(\boldsymbol{\Omega};\boldsymbol{X}_1,\cdots,\boldsymbol{X}_m)=\sum_{k=0}^\infty \sum_{j_1,\cdots,j_m}^k a_{i,k;j_1,\cdots,j_m} \{(\boldsymbol{\Omega}\boldsymbol{X}_1)^{j_1}\cdots (\boldsymbol{\Omega}\boldsymbol{X}_m)^{j_m}\}_{S}
    \label{eq:standnoncomm}
\end{eqnarray}
where $S$ is the chosen ordering.
\end{itemize}
\end{definition}
Clearly, since diagonal matrices $\boldsymbol{X_i}$ commute, defining an ordering for the standard and the mixed commutative maps is unnecessary. As we will see, such difference is important for embeddings of vector dynamical systems in the case of  the mixed commutative map, but not for a  scalar system.
 Notice also that 
\begin{itemize}
    \item the standard commutative map is a linear mix of the dynamical systems functions $f_i(\vec x)$, since a direct calculation shows
\begin{eqnarray}
    \boldsymbol{F}^{(c)}_i(\boldsymbol{X}_1,\cdots,\boldsymbol{X}_m)=\text{diag}\big(f_i(X_{1,1},\cdots,X_{m,1}),\cdots,f_i(X_{1,N},\cdots,X_{m,N}) \big).\label{eq:diagoform}
\end{eqnarray}
which simplifies drastically the evaluation
\item for scalar dynamical systems, the mixed commutative map and the standard non-commutative map reduce to the same quantity
\item in the case of vector dynamical systems, the mixed commutative map preserves the commutativity of the target variables, since diagonal matrices $\boldsymbol{X_i}$ commute among themselves. 
\end{itemize}
As a consequence, for scalar dynamical system we will study only the standard commutative and non-commutative maps, while the result will follow also for the mixed commutative map from the non-commutative one. However, we will have to be more careful in the vector case.

\subsection{Projection operator $\boldsymbol{\Omega}$}
\label{sec:proj}

We provide here a few definitions on the projection operators of size $N$ we will consider in the following.
\begin{definition}\label{def:trivialproj}
A projector $\boldsymbol{\Omega}$ is called \textit{trivial} if rank$(\boldsymbol{\Omega})=N$, or, equivalently, if Span$(\boldsymbol{\Omega})=\mathbb{R}^N$. 
\end{definition}
A simple proof shows that the only trivial projector is the identity matrix $\boldsymbol{I}$.

\begin{definition}\label{def:meanfieldproj}
The \textit{uniform mean-field projector} $\boldsymbol{\Omega}_1$ is defined as the square matrix with elements
\[
\Omega_{1,ij}=\frac{1}{N}
\] 
\end{definition}

Let us consider $\boldsymbol{X}$ to be a diagonal matrix, as in the PEDS embedding procedure. Projection using the uniform mean-field operator yields
\begin{equation}
    \boldsymbol{\Omega}_1\boldsymbol{X}=\frac{1}{N}\begin{pmatrix}
    1 & \cdots & 1\\
    \vdots & \ddots & \vdots\\
    1 & \cdots &1
    \end{pmatrix} \begin{pmatrix}
    X_1 & 0 & 0\\
    \vdots & \ddots & \vdots\\
    0 & 0 & X_N
    \end{pmatrix}=\frac{1}{N}\begin{pmatrix} X_1 & X_2 & \cdots & X_N\\
    \vdots & \vdots & \vdots & \vdots\\
     X_1 & X_2 & \cdots & X_N
    \end{pmatrix}
\end{equation}
therefore, the powers of $\boldsymbol{\Omega} \boldsymbol{X}$ appearing in the PEDS procedure, are neither trivial expressions nor sparse matrices, and indeed contain non linear components in the $X_i$ variables.
\begin{example}\ \\
\rule{\textwidth}{0.05cm}
For $N=2$ we have
\begin{equation}
    \boldsymbol{\Omega}_1 \boldsymbol{X}=\frac{1}{2}\begin{pmatrix}X_1 & X_2\\X_1 & X_2\end{pmatrix}, \ \ \ \ (\boldsymbol{\Omega}_1 \boldsymbol{X})^2=\frac{1}{4} \begin{pmatrix} X_1(X_1+X_2) & X_2(X_1+X_2) \\ X_1(X_1+X_2) & X_2(X_1+X_2)\end{pmatrix}=\langle X \rangle \boldsymbol{\Omega}_1 \boldsymbol{X}
\end{equation}
where $\langle X\rangle=\frac{1}{N} \sum_{j=1}^N X_j$.\\
\rule{\textwidth}{0.05cm}
\end{example}
The previous example can easily be generalized to size $N$, showing that $(\boldsymbol{\Omega}_1 \boldsymbol{X})^k=\langle X\rangle^{k-1} \boldsymbol{\Omega}_1 \boldsymbol{X}$, thus justifying the definition of $\boldsymbol{\Omega}_1$ as the uniform mean-field projector. 

A similar property applies to vectors, as $\boldsymbol{\Omega}_1 \vec X=\langle X\rangle \vec 1$.
\begin{example}\ \\
\rule{\textwidth}{0.05cm}
Consider again the Example \ref{ex:exampleexp}. We can write the embedded system as
\begin{equation}
    \frac{d\vec X}{dt}=a \boldsymbol{\Omega} \vec X-\alpha (\boldsymbol{I}-\boldsymbol{\Omega})\vec X=\boldsymbol{\Omega} (a \boldsymbol{\Omega} \boldsymbol{X})\vec 1-\alpha(\boldsymbol{I}-\boldsymbol{\Omega}) \vec X
\end{equation}
which is in the form of a PEDS \eqref{eq:ext0}, with $\boldsymbol{F}=a\boldsymbol{\Omega}\boldsymbol{X}$, $\vec b=\vec 1$ and $\vec G=-\alpha(\boldsymbol{I}-\boldsymbol{\Omega}) \vec X$.\\
\rule{\textwidth}{0.05cm}
\end{example}

We would like to stress the fact that the PEDS mapping is, in general, highly non-trivial, at least as far as the projection operator is not the trivial one: this condition is required because in this case the matrix powers of the form $(\boldsymbol{\Omega} \boldsymbol{X}_i)^k$ couple all the subspace variables in a nonlinear way. 

On the other hand, for the trivial projector, the  standard decay function is identically zero, and the extended system as well as any extended monomial are ordering independent. In fact, as the diagonal matrices $\boldsymbol{X}_j$ commute  among themselves, we have that 
\begin{eqnarray}
\{(\boldsymbol{\Omega} \boldsymbol{X}_1)^{j_1} \cdots (\boldsymbol{\Omega} \boldsymbol{X}_m)^{j_m}\}_{S}=   (\boldsymbol{X}_1)^{j_1} \cdots (\boldsymbol{X}_m)^{j_m}
\end{eqnarray}
if and only if $\boldsymbol{\Omega}=\boldsymbol{I}$. As a consequence, \eqref{eq:embeddingmap} becomes
\begin{equation}
    \frac{d\vec X_i}{dt}= \boldsymbol{\Omega} \boldsymbol{F}_i(\vec X_i,\cdots,\vec X_m) \vec b_i= \begin{pmatrix} f_i(X_{i,1},\dots,X_{i,m})b_{i,1}\\ \vdots \\f_i(X_{i,N},\dots,X_{i,m})b_{i,N}
 \end{pmatrix} \qquad i=1,\dots,m
\end{equation}
thus showing that the PEDS procedure for the trivial projector decouples into $N$ identical copies of the original system.




\section{Embedding via the uniform mean field projector} \label{sec:umfp}

We derive here in a more rigorous way the key results presented above. We focus on the uniform mean field projector $\boldsymbol{\Omega}_1$, as the proofs are easier to be carried out. Nevertheless, several results are actually valid even for a more general projection operator $\boldsymbol{\Omega}$: these will be explicitly denoted by using the general projector $\boldsymbol{\Omega}$ in place of the uniform mean field operator.

\subsection{Simple case: Scalar target system, embedding without decay function}
We start from the case of a one dimensional target dynamical system
\begin{equation}
\frac{dx}{dt}=f(x)
\label{eq:target}
\end{equation}
where $f(x)$ is analytic, so that
\begin{equation}
    f(x)=\sum_{i=0}^\infty a_i x^i.
    \label{eq:1d}
\end{equation}
Following the PEDS procedure, we introduce the projector operator $\boldsymbol{\Omega}$. The extended variable $\vec X$ is thus an $N$-dimensional vector with components $X_i$, and the matrix map associated to $f$ takes either the standard commutative form \eqref{eq:standcomm} so that
\begin{equation}
\vec F(\vec X)=\sum_{i=0}^\infty a_i \boldsymbol{X}^i \vec 1
    \label{eq:thebanalitylemma}
\end{equation}
where we have chosen $\vec b=\vec 1$, or the standard non-commutative form \eqref{eq:standnoncomm} (we remind that for scalar target systems, the mixed commutative and the standard non-commutative forms coincide)
\begin{equation}
\vec F(\vec X)=\sum_{i=0}^\infty a_i \boldsymbol{\Omega}\boldsymbol{X}^i \vec 1
    \label{eq:thebanalitylemma_noncomm}
\end{equation}

Before we begin our discussion on the embedding, it is worth giving a definition of what we mean when we say that a dynamical system is \textit{contained} in another one.
Taking Fig.~\ref{fig:projectedd} as a reference, let our target system be described by the evolution map (the solution of the dynamical system) $\phi_t^\prime:\mathbb{R}\rightarrow \mathbb{R}$, while the PEDS evolution is instead a map $\phi_t:\mathbb{R}^{N}\rightarrow \mathbb{R}^{N}$. 

\begin{definition}
A dynamical system $A$ of size $N_A$ is \textit{contained} in a dynamical system $B$ of dimensions $N_B> N_A$ if a linear operator $\mathcal P:\mathbb{R}^{N_B}\rightarrow \mathbb{R}^{N_A}$ exists such that, for $\vec X\in \mathbb{R}^{N_B}$
\begin{equation}
    \mathcal P\phi_t(\vec X)=\phi_t^\prime(\vec x), \label{eq:contains}
\end{equation}
where $\vec x$ has size $N_A$.
\end{definition}
Given the definition above, we can now prove the following

\begin{proposition} \label{prop:banality}
\textbf{Banality of mean value.}  Let $\mathcal O=(f(x),\boldsymbol{\Omega},0,\vec 1,N)$ be a PEDS tuple of a target dynamical system as in \eqref{eq:target}, where the matrix map can take either the standard commuting or standard non-commuting forms. Then, the extended dynamical system \eqref{eq:thebanalitylemma} or \eqref{eq:thebanalitylemma_noncomm} contains the dynamics of \eqref{eq:1d} for a generic projection operator $\boldsymbol{\Omega}$ satisfying $\boldsymbol{\Omega} \vec 1\neq \vec 0$.
\end{proposition}

\begin{proof} Let us consider an extended variable $\vec X$ subject to the condition $\vec X=x \vec 1$. Then, as $\boldsymbol{X}=x\vec 1$ and $\boldsymbol{\Omega}^i=\boldsymbol{\Omega}$ $i>0$, for both the standard commuting and non-commuting maps we have:
\begin{equation}
\frac{d}{dt} \vec X= \frac{dx}{dt}\vec 1=\begin{cases} \displaystyle
\sum_{i=0}^\infty a_i x^i \boldsymbol{\Omega}  \vec 1 & \text{commuting map}\\[1ex] \displaystyle
\sum_{i=0}^\infty a_i x^i \boldsymbol{\Omega} \vec 1 & \text{non commuting map}
\end{cases}=a_0 \boldsymbol{\Omega} \vec 1+\sum_{i=1}^\infty a_i x^i \boldsymbol{\Omega} \vec 1 = f(x) \boldsymbol{\Omega} \vec 1 
\end{equation}
Therefore, projecting the previous relation onto the span of $\boldsymbol{\Omega}$, i.e. evaluating $\boldsymbol{\Omega} \frac{d}{dt} \vec X=\boldsymbol{\Omega} \vec F(\boldsymbol{X})$, we obtain
\begin{equation}
    \left(\frac{dx}{dt}-f(x)\right)\boldsymbol{\Omega}\vec 1 =0.
\end{equation}
It follows that as $\boldsymbol{\Omega} \vec 1\neq \vec 0$,  then $\frac{dx}{dt}-f(x)=0$. 
In order to prove that the dynamical system is contained, we can project on any component $i$, obtaining
\begin{eqnarray}
   \omega_i\left( \frac{dx}{dt}-f(x)\right)=0
\end{eqnarray}
if $\omega_i=(\boldsymbol{\Omega} \vec 1)_i\neq 0$,
we have then proven that  an initial condition exists for which \eqref{eq:contains} applies.
\end{proof}

Proposition~\ref{prop:banality} is a warm up for the type of proofs that will follow. It shows that if the initial condition for the variables $\vec X$ are chosen homogeneously, then the extended dynamical system will follow the one dimensional dynamics of \eqref{eq:target}. However, condition $\vec X= x \vec 1$ is a strong requirement for the dynamical system. In principle, a dynamically obtained convergence towards a state of the form $\vec X=x \vec 1$ would be a much better demand. To this aim, we introduce the decay functions.

\subsection{Scalar target system: Enforcing the convergence to the mean via decay functions}
We consider now the following form for the extended system \eqref{eq:ext} based on the uniform mean field projector:
\begin{equation}
   \frac{d\vec X}{dt}= \boldsymbol{\Omega}_1 \vec F(\vec X)-\alpha(\vec X-\langle X \rangle\vec 1)=\boldsymbol{\Omega}_1 \vec F(\vec X)-\alpha(\boldsymbol{I}-\boldsymbol{\Omega_1})\vec X
    \label{eq:thebanalitylemma2}
\end{equation}
where $\langle X \rangle=({1}/{N}) \sum_{i=1}^N X_i$ and $\alpha>0$. Notice that $\vec X-\langle X \rangle \vec 1=(\boldsymbol{I}-\boldsymbol{\Omega}_1) \vec X$ because of the properties of $\boldsymbol{\Omega}_1$ discussed in Sec.~\ref{sec:proj}. The second term on  the right hand side of \eqref{eq:thebanalitylemma2} is an ``elastic" force compelling the extended trajectories to remain  close to the mean. The relative strength of the two addends  determines the behavior of the system.

\begin{proposition} \label{prop:convergence1}\textbf{Convergence to the mean}.
The dynamics of \eqref{eq:thebanalitylemma2} is characterized by the same fixed points, if they exist, as for the target system \eqref{eq:target} both for the standard commuting and non-commuting matrix maps. 
\end{proposition}
\begin{proof} Projecting \eqref{eq:thebanalitylemma2} through $\boldsymbol{\Omega}_1$ and using \eqref{eq:thebanalitylemma} we find
\begin{equation}
   \boldsymbol{\Omega}_1\frac{d\vec X}{dt}=\begin{cases} \displaystyle
   \sum_{i=0}^\infty a_i \boldsymbol{\Omega}_1\boldsymbol{X}^i \vec 1-\alpha\cancel{\boldsymbol{\Omega}_1 (\boldsymbol{I}-\boldsymbol{\Omega}_1) \vec X} & \text{standard commuting map} \\[1ex] \displaystyle
   \sum_{i=0}^\infty a_i \boldsymbol{\Omega}_1(\boldsymbol{ \Omega_1 X})^i \vec 1-\alpha\cancel{\boldsymbol{\Omega}_1 (\boldsymbol{I}-\boldsymbol{\Omega}_1) \vec X} & \text{standard noncommuting map} \\
   \end{cases}
\end{equation}
which reduces to the banality lemma enforcing $\vec X=x \vec 1$. 

Considering the complementary projection, we have
\begin{equation}
  (\boldsymbol{I}- \boldsymbol{\Omega}_1)\frac{d\vec X}{dt}=\cancel{(\boldsymbol{I}-\boldsymbol{\Omega}_1)\boldsymbol{\Omega}_1}\vec F(\vec X)-\alpha( (\boldsymbol{I}-\boldsymbol{\Omega}_1)\vec X -\langle X\rangle \cancel{(\boldsymbol{I}-\boldsymbol{\Omega}_1)\vec 1})
\end{equation}
or, defining $\vec X_{c}=(\boldsymbol{I}-\boldsymbol{\Omega}_1)\vec X$,
\begin{equation}
\frac{d}{dt}\vec X_{c}=-\alpha \vec X_{c}.
\label{eq:tatu}
\end{equation}
Equation \eqref{eq:tatu} represents the dynamics of the $N-1$ modes that make $\vec X$ non-uniform. The above implies that any non-uniform mode of $\vec X$ decays exponentially, and thus $\vec X-\langle X\rangle\vec 1\rightarrow 0$ in a time $t\gg \tau={1}/{\alpha}$. This concludes the proof. 
\end{proof}

In conclusion, using the uniform mean field projector $\boldsymbol{\Omega}_1$ and the standard decay functions as in \eqref{eq:thebanalitylemma2}, $\vec X(t)$ converges to the right mean, and thus to the same fixed points as the target system. Let us now provide some technical results to support the idea that the decay functions project back on the subspace of our interest. The result is in fact not strictly limited to the standard decay functions. We now prove the decay of the modes in Ker$(\boldsymbol{\Omega}_1)$ for a generalized set of decay functions. Let us consider
\begin{equation}
    \frac{d\vec X}{dt}=\boldsymbol{\Omega}_1\vec{ F} -\begin{cases} \boldsymbol{D}(\boldsymbol{I}-\boldsymbol{\Omega}_1) \vec{X} &  \text{generalization A}\\
    (\boldsymbol{I}-\boldsymbol{\Omega}_1)\boldsymbol{D}(\boldsymbol{I}-\boldsymbol{\Omega}_1) \vec{X} & \text{generalization B}
    \end{cases}
    \label{eq:genab1}
\end{equation}
where $\boldsymbol{D}$ is a positive diagonal matrix with diagonal elements $\alpha_1,\dots,\alpha_n>0$. If $\alpha=\alpha_1=\dots=\alpha_n$, both generalizations reduce to the standard decay function. In the general case, their difference becomes evident projecting via $\boldsymbol{\Omega}_1$
\begin{eqnarray}
    \boldsymbol{\Omega}_1\frac{d\vec X}{dt}  =\boldsymbol{\Omega}_1 \vec{F}-\begin{cases}
    \boldsymbol{\Omega}_1 \boldsymbol{D} (\boldsymbol{I}-\boldsymbol{\Omega}_1 )\vec X &   \text{generalization A}\\
    \vec 0 &  \text{generalization B}
    \end{cases}
    \label{eq:genab0}
\end{eqnarray}
i.e., the PEDS embeddings $$\mathcal O_A=(f(x),\boldsymbol{\Omega}_1,-\boldsymbol{D}(\boldsymbol{I}-\boldsymbol{\Omega}_1)\vec X,\vec 1,N)$$ and $$\mathcal O_B=(f(x),\boldsymbol{\Omega}_1,-(\boldsymbol{I}-\boldsymbol{\Omega}_1)\boldsymbol{D}(\boldsymbol{I}-\boldsymbol{\Omega}_1)\vec X,\vec 1,N).$$ 
Clearly, the first part of the proof of the banality lemma remains valid also in these cases. On the other hand, projecting via $(\boldsymbol{I}-\boldsymbol{\Omega}_1)$, we obtain for both generalizations the following governing equation for the non-uniform modes
\begin{equation}
\frac{d\vec X_{c}}{dt}=- (\boldsymbol{I}-\boldsymbol{\Omega}_1) \boldsymbol{D} \vec X_{c}
\label{eq:projab}
\end{equation}
whose solution is 
\begin{eqnarray}
    \vec X_{c}(t)=\left. e^{-(\boldsymbol{I}-\boldsymbol{\Omega}) \boldsymbol{D}t}\vec X_{c}(0)\right|_{\boldsymbol{\Omega}=\boldsymbol{\Omega}_1}.
\end{eqnarray}
We show now that the two generalizations A and B are $\boldsymbol{\Omega}$-eligible, i.e.  that $\vec X_{c}$ asymptotically approaches the zero vector. We prove the following proposition for general projectors:
\begin{proposition} 
\label{prop:span}
Given the governing equation \eqref{eq:projab} written for a general projector $\boldsymbol{\Omega}$, assuming $\vec X_{c}(0)\in \mathrm{Span}(\boldsymbol{I}-\boldsymbol{\Omega})$ then $\vec X_{c}(t)\in \mathrm{Span}(\boldsymbol{I}-\boldsymbol{\Omega})\ \forall t$.
\end{proposition}
\begin{proof} The solution of \eqref{eq:projab} takes the form
\begin{equation}
    \vec X_{c}(t)=e^{\boldsymbol{A} t} \vec X_{c}(0)
\end{equation}
where $\boldsymbol{A}=(\boldsymbol{I}-\boldsymbol{\Omega})\boldsymbol{D}$.
Expanding the exponential, we get
\[
\vec X_{c}(t)=\vec X_{c}(0)+\sum_{k=1}^{+\infty} \frac{t^k}{k!} ((\boldsymbol{I}-\boldsymbol{\Omega})\boldsymbol{D})^k\vec X_{c}(0)
\]
that, projecting through $\boldsymbol{I}-\boldsymbol{\Omega}$, becomes
\begin{equation}
(\boldsymbol{I}-\boldsymbol{\Omega})\vec X_{c}(t)=(\boldsymbol{I}-\boldsymbol{\Omega})\vec X_{c}(0)+\sum_{k=1}^{+\infty} \frac{t^k}{k!} ((\boldsymbol{I}-\boldsymbol{\Omega})\boldsymbol{D})^k\vec X_{c}(0). 
\end{equation}
We notice that if $\vec X_{c}(0)\in \mathrm{Span}(\boldsymbol{I}-\boldsymbol{\Omega})$, then we can express $\vec X_c(0)=\sum_j a_j \vec v_j$ where $\vec v_j$ are eigenvectors associated to the eigenvalue equal to 1 of $\boldsymbol{I}-\boldsymbol{\Omega}$. Thus, $(\boldsymbol{I}-\boldsymbol{\Omega})\vec X_c(0)=\sum_j a_j (\boldsymbol{I}-\boldsymbol{\Omega})\vec v_j=\sum_j a_j \vec v_j=\vec X_c(0)$. This implies that
\begin{equation}
(\boldsymbol{I}-\boldsymbol{\Omega})\vec X_{c}(t)=\vec X_{c}(0)+\sum_{k=1}^{+\infty} \frac{t^k}{k!} ((\boldsymbol{I}-\boldsymbol{\Omega})\boldsymbol{D})^k\vec X_{c}(0)=e^{(\boldsymbol{I}-\boldsymbol{\Omega})\boldsymbol{D}t} \vec X_{c}(0)=\vec X_{c}(t)
\end{equation}
Thus, we have $(\boldsymbol{I}-\boldsymbol{\Omega})\vec X_{c}(t)=\vec X_{c}(t)$, i.e. $\vec X_{c}(t)\in \mathrm{Span}(\boldsymbol{I}-\boldsymbol{\Omega})$. 
\end{proof}
As a result of the proposition above,  vector $\vec X_{c}(t)$ is contained in the subspace spanned by $\boldsymbol{I}-\boldsymbol{\Omega}$ at all times and for any projector, and thus also for $\boldsymbol{\Omega}_1$.

\begin{corollary} \label{cor:lyapunov1}
Equation \eqref{eq:projab} implies that, if $\vec X_{c}(0)\in \mathrm{Span}(\boldsymbol{I}-\boldsymbol{\Omega}_1)$, then in \eqref{eq:genab0} one has $lim_{t\rightarrow \infty } \vec X(t)= \langle X \rangle\vec 1$.
\end{corollary}
\begin{proof}
We consider the dynamics for the modes $\vec X_{c}=(\boldsymbol{I}-\boldsymbol{\Omega}_1) \vec X(t)$ from \eqref{eq:projab}, and we use a Lyapunov stability argument.
Let us consider the following functional: $V(\vec X_{c})=\vec X_{c}\cdot \vec X_{c}\geq 0$. Then, 
\begin{align}
\frac{d}{dt} V&=2\left(\frac{d}{dt}\vec X_{c}(t)\right)\cdot \vec X_{c}(t)\nonumber \\
&=-2((\boldsymbol{I}-\boldsymbol{\Omega}_1) \boldsymbol{D}\vec X_{c}(t))\cdot \vec X_{c}(t)\nonumber \\
&=-2( \sqrt{\boldsymbol{D}}\vec X_{c}(t))\cdot \sqrt{\boldsymbol{D}} (\boldsymbol{I}-\boldsymbol{\Omega}_1)^T\vec X_{c}(t)\nonumber \\
&=-2( \sqrt{\boldsymbol{D}} \vec X_{c}(t))\cdot \sqrt{\boldsymbol{D}} (\boldsymbol{I}-\boldsymbol{\Omega}_1)\vec X_{c}(t).
\end{align}
Using the fact that $\vec X_{c}(t)\in \text{Span}(\boldsymbol{I}-\boldsymbol{\Omega}_1)$ from the previous Proposition, we obtain that
\begin{equation}
\frac{d}{dt} V=-2 || \sqrt{\boldsymbol{D}} \vec X_{c}(t)||^2 \leq 0.
\end{equation}
Since the only minimum of $V(\vec X)$ is $\vec X=\vec 0$, then $\vec X_{c}(t)\rightarrow \vec 0$ for $t\rightarrow \infty$.
This proves that $\vec X_{c}(0)\rightarrow \vec 0 $, and thus $\vec X\rightarrow \langle X\rangle \vec 1$, for $t\rightarrow \infty$. 
\end{proof}

Propositions~\ref{prop:banality}, \ref{prop:convergence1} and \ref{prop:span}, along with Corollary~\ref{cor:lyapunov1} imply that for a one-dimensional dynamical system, the PEDS $\mathcal O=(f(x),\boldsymbol{\Omega}_1,-(\boldsymbol{I}-\boldsymbol{\Omega}_1)\vec X,\vec 1,N)$ contains the fixed points of the original dynamical system. In particular, Corollary \ref{cor:lyapunov1} implies that the extended system converges to  an asymptotic state of the form $\vec X(t)=x(t) \vec 1$. Therefore, through the banality of the mean value Lemma,  the PEDS embedding will contain the original dynamical system. 

For practical purposes, it is sufficient to consider the observable $\tilde x(t)=\langle X\rangle =\frac{1}{N}\sum_{i,j=1}^N {\Omega}_{1,ij} X_j(t)$ in order to recover the location of the fixed points. Clearly, this example applies only to a one-dimensional dynamical system. However, the result can be extended to the vector case following similar considerations.

\begin{example}\ \label{ex:2D} \\
\rule{\textwidth}{0.05cm}
As an example of a one dimensional dynamical system embedded in $N=2$ dimensions, let us consider the dynamical system 
\begin{eqnarray}
    \frac{dx}{dt}=x-x^2 \label{eq:example2}
\end{eqnarray}
whose stable fixed point is given by $x^*=1$. 
A PEDS embedding $\mathcal O=(x-x^2,\boldsymbol{\Omega}_1,-\alpha(\boldsymbol{I}-\boldsymbol{\Omega}_1) \vec X,\vec 1,2)$ is given by the two coupled differential equations:
\begin{eqnarray}
    \frac{d\vec X}{dt}=\boldsymbol{\Omega}_1 \left( \boldsymbol{X} -\boldsymbol{X}^2\right) \vec 1-\alpha( \boldsymbol{I}-\boldsymbol{\Omega}_1)\vec X
\end{eqnarray}
whose components can be made explicit introducing $\langle X\rangle=\frac{1}{2}(X_1+X_2)$ and evaluating the matrix expressions. The result is
\begin{eqnarray}
    \frac{dX_1}{dt}=\langle X\rangle-\langle X\rangle^2-\alpha (X_1-\langle X\rangle) \\[1ex]
    \frac{dX_2}{dt}=\langle X\rangle-\langle X\rangle^2-\alpha (X_2-\langle X\rangle) 
\end{eqnarray}
whose stable fixed point is easily seen to be  $X^*_1=X^*_2=1$.\\
\rule{\textwidth}{0.05cm}
\end{example}

Example~\ref{ex:2D}  highlights an important property of the uniform mean field embedding. We gain some intuition on how the uniform mean field projector works by exploiting a direct evaluation of the matrix powers. For instance, for  the PEDS $\mathcal O=(f(x),\boldsymbol{\Omega}_1,-\alpha (\boldsymbol{I}-\boldsymbol{\Omega}_1) \vec X,\vec 1,N)$, then
\begin{eqnarray}
    \frac{d\vec X}{dt}=\boldsymbol{\Omega}_1 \vec F(\vec X)-\alpha (\boldsymbol{I}-\boldsymbol{\Omega}_1) \vec X.
    \label{eq:nne}
\end{eqnarray}
Using the properties
\[
(\boldsymbol{\Omega}_1 \boldsymbol{X})^k=\langle X\rangle^{k-1} \boldsymbol{\Omega}_1 \boldsymbol{X}\quad k\ge 1, \qquad \boldsymbol{\Omega}_1 \boldsymbol{X} \vec 1=\boldsymbol{\Omega}_1 \vec X=\langle X \rangle \vec 1, \qquad \boldsymbol{\Omega}_1  \vec 1= \vec 1
\]
we obtain the equivalent form for \eqref{eq:nne}
\begin{eqnarray}
    \frac{d \vec X}{dt}= f(\langle x\rangle) \vec 1-\alpha (\vec X-\langle x\rangle \vec 1).
\end{eqnarray}
Multiplying on the left times  $\boldsymbol{\Omega}_1$ and times $\boldsymbol{I}-\boldsymbol{\Omega}_1$, yields
\begin{align}
    \frac{d}{dt} \left(\langle X\rangle \vec 1\right)&=f(\langle X \rangle) \vec 1, \label{eq:meanvalue1}\\
    \frac{d}{dt} \left(\vec X-\langle X\rangle \vec 1\right)&=-\alpha (\vec X-\langle X\rangle \vec 1). \label{eq:meanvalue2}
\end{align}
Therefore,  the mean value $\langle X\rangle$ in \eqref{eq:meanvalue1} follows exactly the target system  dynamics, while \eqref{eq:meanvalue2} asymptotically determines $\vec X\rightarrow \langle X\rangle \vec 1$ for $t\gg 1/\alpha$. It follows that for scalar target systems, because of the identity $(\boldsymbol{\Omega_1 X})^k=\langle X\rangle^{k-1}  \boldsymbol{\Omega_1 X}$, the standard commutative and non-commutative maps are identical. This is no longer true for vector target systems, as we discuss below.

\subsection{General case: Vector target system}

We will now focus on the generalization of the previous results to the case of a vector target dynamical system of arbitrary dimension. Such generalization is involved, basically because  the ordering $S$ becomes important (at least for the standard non-commutative map) and this depends on the fact that matrix terms such as $(\boldsymbol{\Omega} \boldsymbol{X})^i$ and  $(\boldsymbol{\Omega} \boldsymbol{X})^j$ do not commute. Nevertheless, exploiting the properties of the uniform mean field projector $\boldsymbol{\Omega}_1$ certain exact results can be obtained. 

We consider an $m$ dimensional target system as in \eqref{eq:orig}, where a Taylor expansion of the defining functions $f_i(x_1,\dots,x_m)$ takes the form \eqref{eq:taylormonomial} after choosing the ordering $S$.
Following the PEDS procedure, we construct as a first instance the embedding map $\mathcal O=(\{f_i\},\boldsymbol{\Omega}_1,0,S,\vec 1,N)$ in the abscence of decay functions.
Given the extended variables $\vec X_s$ ($s=1,\dots,m$) of size $N$, we write the embedding maps \eqref{eq:standcomm}--\eqref{eq:standnoncomm}  as
\begin{subequations}
\label{eq:pedsmulti}
\begin{align}
\frac{d \vec X_s}{dt}&=\boldsymbol{\Omega}_1\sum_{k=0}^\infty \sum_{i_1,\dots,i_m }^k a_{s,k;i_1\dots i_m} \boldsymbol{X}_1^{i_1}\dots \boldsymbol{X}_m^{i_m} \vec 1 \nonumber\\[1ex]
&\qquad\qquad\qquad\qquad\qquad\qquad\qquad  \text{standard commuting map} \label{eq:pedsmultic}\\[1ex]
\frac{d \vec X_s}{dt}&=a_{s,0}\boldsymbol{\Omega}_1 +
\boldsymbol{\Omega}_1\sum_{k=1}^\infty \sum_{i_1,\cdots,i_m}^k a_{s,k;i_1,\cdots,i_m} (\boldsymbol{\Omega}_1(\boldsymbol{X}_1^{i_1}\cdots \boldsymbol{X}_m^{i_m})^{1/k})^k\vec 1 \nonumber\\[1ex]
&\qquad\qquad\qquad\qquad\qquad\qquad\qquad \text{mixed commuting map} \label{eq:pedsmultimc}\\[1ex]
\frac{d \vec X_s}{dt}&=\boldsymbol{\Omega}_1\sum_{k=0}^\infty \sum_{i_1,\dots,i_m }^k a_{s,k;i_1\dots i_m} \{(\boldsymbol{\Omega}_1 \boldsymbol{X}_1)^{i_1}\dots (\boldsymbol{\Omega}_1 \boldsymbol{X}_m)^{i_m} \}_S\vec 1 \nonumber\\[1ex]
&\qquad\qquad\qquad\qquad\qquad\qquad\qquad  \text{standard noncommuting map} \label{eq:pedsmultinc}
\end{align}
\end{subequations}
where $S$ is the chosen ordering of the variables. The question is whether also in this case the banality of the mean value Proposition~\ref{prop:banality} still holds.

An easy proof shows that assuming a solution $\vec X_s=x_s \vec 1$, the \textit{banality of the mean value} lemma applies also in the vector case, irrespective of the chosen extended variable ordering. 

\begin{corollary} \textbf{Multivariate banality of the mean value}. \label{coro:multibanality}
Let $\boldsymbol{\Omega}_1$ be the uniform mean field projector, and $\mathcal O=(\{f_i\},\boldsymbol{\Omega}_1,\{-\alpha (\boldsymbol{I}-\boldsymbol{\Omega}_1)\vec X_i\}, S,\vec 1,N)$ the embedding with $S$ an arbitrary ordering. Following the PEDS procedure, the dynamics of the $m$ variables $\vec X_i$ is determined by
\begin{eqnarray}
    \frac{d\vec X_i}{dt}=\boldsymbol{\Omega}_1 \vec F_i(\vec X_1,\cdots, \vec X_m)-\alpha(\boldsymbol{I}-\boldsymbol{\Omega}_1) \vec X_i
    \label{eq:pedsomega1}
\end{eqnarray}
where $\vec F_i=\boldsymbol{F}_i(\vec X_1,\cdots, \vec X_m)\vec 1$ is the vector constructed following the PEDS procedure defined in \eqref{eq:pedsmulti}. 
We define the projection operator $\mathcal P_{\boldsymbol{\Omega}_1}=\frac{1}{N} \vec 1^T$, and the projected variables
 $\langle X_i(t)\rangle=\mathcal P_{\boldsymbol{\Omega}_1}    \vec X_i(t)$.  Then, the following two statements hold true
\begin{equation}
   (a)\ \ \  \frac{d}{dt} \langle X_i\rangle= 0 \Longrightarrow \frac{d}{dt}  x_i=0,
\end{equation}
and (b) if the extended system approaches a fixed point $\vec X_i^*$ for times $t\gg {1}/{\alpha}$, then the projection $\mathcal P_{\boldsymbol{\Omega}_1}    \vec X_i^*$ is a fixed point of the target system.
\end{corollary}
\begin{proof} Let us first prove statement $(a)$, which is a corollary of the banality of the mean value Proposition~\ref{prop:banality}.
We set $\vec X_i(t)=\langle X_i(t)\rangle \vec 1$. In all cases of the standard commutative, mixed commutative and non-commutative maps, the standard decay function is identically zero, and it is not hard to see that, for an arbitrary orderings $S$, we have in all cases the same expression for each term of the expansion:
\begin{align}
    \left.\begin{array}{c}
        \displaystyle \boldsymbol{\Omega}_1 (\boldsymbol{X}_1)^{i_1}\cdots (\boldsymbol{X}_m)^{i_m}\vec 1  \\[1ex]
        \displaystyle \left(\boldsymbol{\Omega}_1  (\boldsymbol{X}_1)^{\frac{i_1}{\sum_{k}i_k}}\cdots (\boldsymbol{X}_m)^{\frac{i_m}{\sum_{k}i_k}}\right)^{\sum_k i_k}  \vec 1\\[2ex]
        \displaystyle\{(\boldsymbol{\Omega}_1  \boldsymbol{X}_1)^{i_1}\cdots (\boldsymbol{\Omega}_1  \boldsymbol{X}_m)^{i_m} \}_S\vec 1
    \end{array}\right\}
    &=\langle X_1(t)\rangle^{i_1}\cdots \langle X_m(t)\rangle^{i_m} \boldsymbol{\Omega}_1 \vec 1 \nonumber\\[1ex]
    &=\langle X_1(t)\rangle^{i_1}\cdots \langle X_m(t)\rangle^{i_m} \vec 1
\end{align}
As a consequence, a relatively simple calculation shows
\begin{eqnarray}
    \frac{d}{dt} \vec X_i=\frac{d}{dt} \langle X_i(t)\rangle \vec 1=f_i(\langle X_1(t)\rangle,\dots,\langle X_m(t)\rangle) \boldsymbol{\Omega}_1 \vec 1=f_i(\langle X_1(t)\rangle,\dots,\langle X_m(t)\rangle) \vec 1,
\end{eqnarray}
or
\begin{eqnarray}
    \left(\frac{d}{dt}\langle X_i(t)\rangle-f_i(\langle X_1(t)\rangle,\dots,\langle X_m(t)\rangle) \right)\vec 1=0.
\end{eqnarray}
Replacing $\langle X_i(t)\rangle$ with $x_i(t)$, we obtain that the fixed points of the extended system under the assumption $\vec X_i(t)=\langle X_i(t)\rangle \vec 1$ must be the same as for the target dynamical system.

We now turn to statement $(b)$. The initial condition $\vec X_i(t=0)$ is now arbitrary. Multiplying on the left \eqref{eq:pedsomega1} times $(\boldsymbol{I}-\boldsymbol{\Omega}_1)$, we get
\begin{equation}
    (\boldsymbol{I}-\boldsymbol{\Omega}_1)\frac{d\vec X_i}{dt}=-\alpha(\boldsymbol{I}-\boldsymbol{\Omega}_1) \vec X_i
    \label{eq:proj}
\end{equation}
where we define $\vec X_{i,c}=(\boldsymbol{I}-\boldsymbol{\Omega}_1) \vec X_i$. As Span$(\boldsymbol{I}-\boldsymbol{\Omega}_1)=\mathrm{Ker}(\boldsymbol{\Omega}_1)$, $\vec X_{i,c}$ can be interpreted as a deviation from the average, since $\vec X_i=\langle X_i\rangle \vec 1+\vec X_{i,c}$.  Following almost the same steps as in the proof of the convergence of the mean for the one dimensional system, we arrive at
\begin{gather}
   \frac{d\vec X_{i,c}}{dt}= -\alpha \vec X_{i,c}\implies\nonumber\\
       \left( \langle X_i(t)\rangle-f_i(\langle X_1(t)\rangle,\dots,\langle X_m(t)\rangle)\right)  \vec 1\nonumber\\
       =\Big( \mathcal P_{\boldsymbol{\Omega}_1} \vec X_i(t)-f_i(\mathcal P_{\boldsymbol{\Omega}_1}  \vec X_1(t),\dots,\mathcal P_{\boldsymbol{\Omega}_1}  \vec X_m(t))\Big)  \vec 1\rightarrow \vec 0,
       \text{ for } t\gg \frac{1}{\alpha}.
\end{gather}
where the second expression follows from the first being $\alpha>0$, so that $\vec X_{i,c}(t)\rightarrow \vec 0$ for $t\gg1/\alpha$. Essentially, this implies that the extended system fixed points are those of the target dynamical system: for long enough times the system converges exponentially to the mean in each variable, for which the banality lemma applies. Thus,  if the projected PEDS given by $\mathcal O$ approaches a fixed point, it has to be a fixed point of the target system. Alternatively, the system must not converge. 
\end{proof}

The results of this section show that at least one type of PEDS exists which preserves the fixed points of the target system, thus justifying the entire construction  of the PEDS embedding. We now focus the attention on  how the PEDS procedure modifies the properties of the fixed points, by analyzing their stability. Therefore, we need to look at the spectral properties of the Jacobian at the embedding fixed points.

\subsection{Properties of the Jacobian and fixed points}
Let us now investigate the properties of the Jacobian.

We begin with the one dimensional case
\begin{eqnarray}
    \frac{d\vec X}{dt}=\underbrace{\boldsymbol{\Omega}_1 \boldsymbol{F} \vec 1}_{1}-\underbrace{\alpha (\boldsymbol{I}-\boldsymbol{\Omega}_1) \vec X}_2\label{eq:zeros}
\end{eqnarray}
For the sake of generality, we consider also the generalized decay functions in \eqref{eq:genab1}
\begin{eqnarray}
    \frac{d\vec X}{dt}=\underbrace{\boldsymbol{\Omega}_1 \boldsymbol{F} \vec 1}_{1}-\begin{cases}\underbrace{ \boldsymbol{D}(\boldsymbol{I}-\boldsymbol{\Omega}_1)  \vec X}_{2A} & \text{generalization A}\\
    \underbrace{ (\boldsymbol{I}-\boldsymbol{\Omega}_1)\boldsymbol{D}(\boldsymbol{I}-\boldsymbol{\Omega}_1) \vec X}_{2B} & \text{generalization B}
    \end{cases}
    \label{eq:genab}
\end{eqnarray}
As usual, we shall derive the  results for a general projector $\boldsymbol{\Omega}$ whenever possible.

\subsubsection{Simple case: scalar target system}
We will focus first on the PEDS of a one dimensional target system, as in this case the ordering is immaterial and proofs are easier to carry out.
Initially we consider the PEDS map $\mathcal O=(f(x),\boldsymbol{\Omega}_1,-\alpha (\boldsymbol{I}-\boldsymbol{\Omega}_1)\vec X,\vec b,N)$, i.e. the standard decay function. The embedding thus takes the form
\begin{eqnarray}
    \frac{d}{dt} X_i=(\boldsymbol{\Omega}_1\boldsymbol{F} \vec b-\alpha(\boldsymbol{I}-\boldsymbol{\Omega}_1) \vec X)_i=M_i.
\end{eqnarray}
We prove the following:
\begin{proposition} \label{prop:onepedsjac}\textbf{One dimensional PEDS Jacobian}\\
For a scalar target system, a PEDS of the form $\mathcal O=(f(x),\boldsymbol{\Omega}, -\alpha(\boldsymbol{I}-\boldsymbol{\Omega})\vec X,\vec b,N)$ is characterized by the following functional form of the Jacobian 
\begin{itemize}
    \item for the standard non-commuting map
    \begin{equation}
    {J}^{(nc)}_{ir}=\sum_{j=1}^N\sum_{l=1}^N{\Omega}_{il} \sum_{z=1}^\infty a_z \sum_{k=0}^{z-1} \sum_{s=1}^N (\boldsymbol{\Omega} \boldsymbol{X})^k_{ls} {\Omega}_{sr} (\boldsymbol{\Omega} \boldsymbol{X})^{z-1-k}_{rj} b_j-\alpha (\boldsymbol{I}-\boldsymbol{\Omega})_{ir},
\end{equation}
\item for the standard commuting map
\begin{eqnarray}
{J}^{(c)}_{ir}={\Omega}_{ir} f^\prime(X_r)b_r-\alpha (\boldsymbol{I}-\boldsymbol{\Omega})_{ir}.
\end{eqnarray}
\end{itemize}
\end{proposition}

\begin{proof}
The Jacobian elements are defined as
\begin{eqnarray}
    {J}_{ir}=\frac{\partial {M}_i}{\partial X_r}=\frac{\partial}{\partial X_r}\sum_{j=1}^N (\boldsymbol{\Omega F})_{ij}b_j-\alpha \frac{\partial}{\partial X_r}((\boldsymbol{I}-\boldsymbol{\Omega})\vec X)_i.
\end{eqnarray}

As we need to evaluate the derivatives of the matrix maps, we consider first the derivatives of the matrix quantities depending on $\boldsymbol X$.
We have
\begin{eqnarray}
    \frac{\partial}{\partial X_r} (\boldsymbol{\Omega} \boldsymbol{X})^m_{ij}&=&\sum_{k=0}^{m-1} \sum_{s=1}^N (\boldsymbol{\Omega} \boldsymbol{X})^k_{is} \sum_{l,t=1}^N\frac{\partial ({\Omega}_{sl} X_l \delta_{lt})}{\partial X_r} (\boldsymbol{\Omega} \boldsymbol{X})^{m-1-k}_{tj} \nonumber \\
    &=&\sum_{k=0}^{m-1} \sum_{s=1}^N (\boldsymbol{\Omega} \boldsymbol{X})^k_{is} {\Omega}_{sr} (\boldsymbol{\Omega} \boldsymbol{X})^{m-1-k}_{rj},
\end{eqnarray}
where a matrix to zero power coincides with the identity matrix.

Taking into account definition \eqref{eq:standnoncomm} we find
\begin{eqnarray}
    \frac{\partial {F}_{ij}^{(nc)}}{\partial X_r}&=& \sum_{l=1}^N{\Omega}_{il}\frac{\partial}{\partial {X_r}} \sum_{z=1}^\infty a_z (\boldsymbol{\Omega} \boldsymbol{X})^z_{lj}\nonumber\\
    &=&\sum_{l=1}^N{\Omega}_{il} \sum_{z=1}^\infty a_z \sum_{k=0}^{z-1} \sum_{s=1}^N (\boldsymbol{\Omega} \boldsymbol{X})^k_{ls} {\Omega}_{sr} (\boldsymbol{\Omega} \boldsymbol{X})^{z-1-k}_{rj}
    \label{eq:jac}
\end{eqnarray}
Thus, the first term of the Jacobian is simply given by
\begin{eqnarray}
   {J}^{(nc,1)}_{ir}=\sum_{j=1}^N\sum_{l=1}^N {\Omega}_{il} \sum_{z=1}^\infty a_z \sum_{k=0}^{z-1} \sum_{s=1}^N (\boldsymbol{\Omega} \boldsymbol{X})^k_{ls} {\Omega}_{sr} (\boldsymbol{\Omega} \boldsymbol{X})^{z-1-k}_{rj} b_j.
\end{eqnarray}
In the case of the standard commutative map, the result can be easily derived exploiting \eqref{eq:diagoform} if function $f(x)$ is known in closed form. On the other hand, making use of the power expansion of the function, we can directly calculate the derivatives noticing that
\begin{eqnarray}
    \frac{\partial}{\partial X_r} {X}^k_{ij} = \frac{\partial}{\partial X_r} (\text{diag}(X_1^k,\cdots,X_m^k))_{ij}= (\text{diag}(0,\dots, k X_r^{k-1},\dots,0))_{ij}=k \delta_{ij}\delta_{ir} X_r^{k-1}
\end{eqnarray}
so that
\begin{eqnarray}
    {J}^{(c,1)}_{ir}&=&\sum_{j=1}^N\frac{\partial \left(\boldsymbol{\Omega} \boldsymbol{F}^{(c)}\right)_{ij}}{\partial X_r}b_j \nonumber \\
    &=&\sum_{j=1}^N\sum_{l=1}^N {\Omega}_{il} \sum_{z=0}^\infty a_z \frac{\partial {X}^z_{lj}}{\partial X_r} b_j \nonumber \\
    &=&\sum_{j=1}^N\sum_{l=1}^N {\Omega}_{il} \sum_{z=1}^\infty z a_z \delta_{lj}\delta_{lr} X_r^{z-1}b_j \nonumber \\
    &=&{\Omega}_{ir} f^\prime(X_r)b_r
\end{eqnarray}
where $f^\prime(x)$ denotes the derivative of $f(x)$.

Similarly, for the standard decay function $\boldsymbol{D}=\alpha \boldsymbol{I}$, it is not hard to see that the second term of the Jacobian is the same irrespective of the chosen matrix embedding
\begin{eqnarray}
{J}^{(2)}_{ir}&=&-\alpha \sum_{j=1}^N \frac{\partial}{\partial x_r} \Big((\boldsymbol{I}-\boldsymbol{\Omega})\boldsymbol{X}\Big)_{ij}\nonumber \\
&=&-\alpha \sum_{j=1}^N(\boldsymbol{I}-\boldsymbol{\Omega})_{ir}\delta_{rj}=-\alpha(\boldsymbol{I}-\boldsymbol{\Omega})_{ir}
\label{eq:jacsecondpart}
\end{eqnarray}
Summing $\boldsymbol{J}^{(1)}$ and $\boldsymbol{J}^{(2)}$, we find the expression to be proven.
\end{proof}

As a direct Corollary of Proposition \ref{prop:onepedsjac}, we obtain that for $\boldsymbol{\Omega}=\boldsymbol{\Omega}_1$ and $\vec b=\vec 1$, the Jacobian takes a simpler form.

\begin{corollary}\label{cor:onedimpedsmf}
Consider the uniform mean field PEDS with standard decay function $\mathcal O=(f(x),\boldsymbol{\Omega}_1, -\alpha(\boldsymbol{I}-\boldsymbol{\Omega}_1)\vec X,\vec 1,N)$ of a scalar dynamical system characterized by a fixed point $x^\ast$. The Jacobian of the PEDS in its fixed point $\vec X^\ast=x^\ast\vec 1$ is given by
\begin{equation}
    {J}_{ir}(\vec X^\ast)=-\alpha {\delta}_{ir}+\frac{1}{N}(f'(x^\ast)+\alpha)  \Longleftrightarrow \boldsymbol{J}(\vec X^\ast)=-\alpha\boldsymbol{I}+(f'(x^\ast)+\alpha)\boldsymbol{\Omega}_1,
\end{equation}
both for the standard commutative and non-commutative maps.
\end{corollary}
\begin{proof}
We exploit Proposition~\ref{prop:onepedsjac}, considering $\boldsymbol{\Omega}=\boldsymbol{\Omega}_1$ and $\vec b=\vec 1$, from which, because of Proposition \ref{prop:banality}, we have $\boldsymbol{X}^\ast=x^\ast \boldsymbol{I}$. Substituting into \eqref{eq:jac}, we obtain
\begin{eqnarray}
    {J}^{(nc,1)}_{ir}(\vec X^\ast)&=&\sum_{j=1}^N\sum_{l=1}^N {\Omega}_{1,il} \sum_{z=1}^\infty a_z (x^\ast)^{z-1}\sum_{k=0}^{z-1} \sum_{s=1}^N \Omega^k_{1,ls} {\Omega}_{1,sr} \Omega^{z-1-k}_{1,rj}\nonumber \\
    &=&\sum_{j=1}^N \sum_{z=1}^\infty a_z (x^\ast)^{z-1}\sum_{k=0}^{z-1} \sum_{s=1}^N {\Omega}_{1,is} {\Omega}_{1,sr} \sum_{j=1}^N {\Omega}^{z-1-k}_{1,rj}
\end{eqnarray}
Since $\boldsymbol{I} \vec 1=\boldsymbol{\Omega}_1 \vec 1=\vec 1$, we find 
\begin{eqnarray}
    {J}^{(nc,1)}_{ir}(\vec X^\ast)&=& \sum_{z=1}^\infty a_z (x^\ast)^{z-1}\sum_{k=0}^{z-1} {\Omega}_{1,ir}\nonumber\\
    &=& \sum_{z=1}^\infty z a_z (x^\ast)^{z-1} {\Omega}_{1,ir}=f^\prime(x^\ast) {\Omega}_{1,ir}
\end{eqnarray}
therefore, in these conditions the first part of the Jacobian takes the same expression as for the standard commutative map. Summing the second term \eqref{eq:jacsecondpart}, finally yields for both maps
\begin{eqnarray}
    {J}_{ir}(\vec X^\ast)
    &=&f^\prime(x^\ast){\Omega}_{1,ir}-\alpha (\boldsymbol{I}-\boldsymbol{\Omega}_1)_{ir}=-\alpha {\delta}_{ir}+(f^\prime(x^\ast)+\alpha){\Omega}_{1,ir}\nonumber \\
    &=&-\alpha {\delta}_{ir}+\frac{1}{N}(f^\prime(x^\ast)+\alpha).
    \label{eq:jac1}
\end{eqnarray}

\end{proof}

At this point we can start to draw some partial conclusions. In fact, we can use the properties of the fixed points of the target system to understand how these are transformed by the embedding procedure.
The target system equilibrium is unstable if $f^\prime(x^\ast)>0$, while it is stable for $f^\prime(x^\ast)<0$. Finally, $f^\prime(x^\ast)=0$ corresponds to a saddle. Since any scalar dynamical system is conservative, it can be expressed as
\begin{eqnarray}
    \frac{dx}{dt}=f(x)=-\frac{\partial V(x)}{\partial x},
\end{eqnarray}
where $V$ is a potential function. The extrema $x^\ast$ of $V(x)$ correspond to minima, maxima and saddles, as we show in Fig. \ref{fig:onedfig}.

For the PEDS $\mathcal O=(f(x),\boldsymbol{\Omega}_1,-\alpha (\boldsymbol{I}-\boldsymbol{\Omega}_1)\vec X,\vec 1,N)$, Corollary \ref{cor:onedimpedsmf} implies that the Jacobian spectrum follows from the spectral properties of $\boldsymbol{\Omega}_1$. In fact, $\boldsymbol{\Omega}_1$ has one eigenvalue equal to $1$, and $N-1$ identical, null eigenvalues. Thus,  the spectrum $\Lambda$ of the Jacobian at $\vec X^*=x^\ast \vec 1$ is given by
\begin{eqnarray}
    \Lambda(\boldsymbol{J}(x^\ast))=\{\{-\alpha\}_{N-1},\{f^\prime(x^\ast)\}_1\},
\end{eqnarray}
i.e. $-\alpha$ with multiplicity $N-1$, and $f'(x^\ast)$ with multiplicity 1.

\begin{figure}
    \centering
    \includegraphics[scale=0.5]{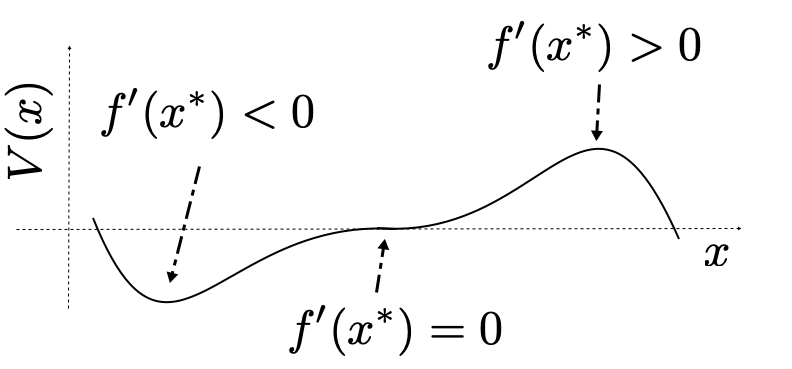}
    \caption{Interpreting the one dimensional dynamical system as the gradient of a potential, via $f(x)=-\partial V(x)/\partial x$.}
    \label{fig:onedfig}
\end{figure}

We can therefore carry out a stability analysis of the PEDS fixed point $\vec X^\ast$ as follows:
\begin{eqnarray}
    \vec X^*~\text{is}~\begin{cases}
    \text{stable} & \text{if } x^*\text{ is stable},\\
    \text{a saddle point} & \text{if } x^*\text{ is a saddle point}, \\
    \text{a saddle point} & \text{if } x^*\text{ is unstable}.
    \end{cases}
    \label{eq:1}
\end{eqnarray}
We thus see what are the benefits of the PEDS from the point of view of the target system. While in the scalar case ``barriers" can be present, these can be made to disappear via the PEDS. Although this specific feature is peculiar to scalar target systems, this result will later be useful also for vector target systems. A graphical representation is shown in Fig. \ref{fig:onedfig2}.

We can conclude that the PEDS procedure $\mathcal O=(f(x),\boldsymbol{\Omega}_1,-\alpha (\boldsymbol{I}-\boldsymbol{\Omega}_1)\vec X,\vec 1,N)$ preserves stable and saddle fixed points of the target dynamics, while it turns unstable fixed points into saddle points.

 Let us now consider the Jacobian properties in presence of the generalized decay functions in \eqref{eq:genab}. A simple generalization of \eqref{eq:jacsecondpart} yields
\begin{eqnarray}
    {J}^{(2)}_{ir}(\vec X^*)=\begin{cases}
    -\Big(\boldsymbol{D}(\boldsymbol{I}-\boldsymbol{\Omega})\Big)_{ir} & \text{generalization A}\\[1ex]
    -\Big((\boldsymbol{I}-\boldsymbol{\Omega})\boldsymbol{D}(\boldsymbol{I}-\boldsymbol{\Omega})\Big)_{ir} & \text{generalization B}.
    \end{cases}
\end{eqnarray}
Therefore, the full Jacobian for $\boldsymbol{\Omega}=\boldsymbol{\Omega}_1$ and $\vec b=\vec 1$, becomes 
\begin{eqnarray}
    \boldsymbol{J}(\vec X^*)=f'(x^*) \boldsymbol{\Omega}_1 - \begin{cases}
    \boldsymbol{D}(\boldsymbol{I}-\boldsymbol{\Omega}_1) & \text{generalization A}\\[1ex]
    (\boldsymbol{I}-\boldsymbol{\Omega}_1)\boldsymbol{D}(\boldsymbol{I}-\boldsymbol{\Omega}_1) & \text{generalization B}.
    \end{cases}
\end{eqnarray}

We wish, now, to determine the spectrum  $\Lambda(\boldsymbol{J}(\vec X^*))$. 
We consider the two generalizations separately:
\begin{itemize}
    \item Generalization B. As a consequence of the identities $\boldsymbol{\Omega}(\boldsymbol{I}-\boldsymbol{\Omega})=(\boldsymbol{I}-\boldsymbol{\Omega})\boldsymbol{\Omega}=\boldsymbol{0}$, we can deduce $[\boldsymbol{J}^{(1)},\boldsymbol{J}^{(2)}]=\boldsymbol{0}$ $\forall \boldsymbol{D}$. Thus, both Jacobian components can be diagonalized in the same basis assembled in matrix $\boldsymbol{T}$, such that $\boldsymbol{T}\boldsymbol{\Omega} \boldsymbol{T}^{-1}=\boldsymbol{D}_{1}$ and $\boldsymbol{T}(\boldsymbol{I}-\boldsymbol{\Omega})\boldsymbol{D} (\boldsymbol{I}-\boldsymbol{\Omega}) \boldsymbol{T}^{-1} =\boldsymbol{D}_{2}$. Then the eigenvalues are given by the elements of the diagonal matrix $\boldsymbol{D}_1+\boldsymbol{D}_2$. Since it is not hard to see that $\text{Span}(\boldsymbol{\Omega})= \text{Ker}((\boldsymbol{I}-\boldsymbol{\Omega})\boldsymbol{D}(\boldsymbol{I}-\boldsymbol{\Omega})) $ and $\text{Span}(\boldsymbol{\Omega})\cup \text{Span}((\boldsymbol{I}-\boldsymbol{\Omega})\boldsymbol{D}(\boldsymbol{I}-\boldsymbol{\Omega}))=\mathbb{R}^N $ with $\text{Span}(\boldsymbol{\Omega})\cap \text{Span}((\boldsymbol{I}-\boldsymbol{\Omega})\boldsymbol{D}(\boldsymbol{I}-\boldsymbol{\Omega}))=\emptyset$, we can focus on the eigenvalues of the two addends of $\boldsymbol{J}$. For $\boldsymbol{J}^{(1)}$, there are $M$ eigenvalues equal to 0 and $N-M$ identical eigenvalues $\lambda=f'(x^*)$. For $\boldsymbol{J}^{(2B)}=(\boldsymbol{I}-\boldsymbol{\Omega})\boldsymbol{D}(\boldsymbol{I}-\boldsymbol{\Omega})$, there are $N-M$ null eigenvalues, while the remaining $M$ eigenvalues satisfy $0 < \lambda\leq \max \{D_{ii}\}$. Clearly, $M$ is determined by the cardinality of $\text{Span}(\boldsymbol{\Omega})$, equal to 1 for $\boldsymbol{\Omega}_1$.

    \item Generalization A. This case is slightly more complicated, since the Jacobian is not symmetric and, thus, its eigenvalues can be complex. We can provide some results exploiting  Gerschgorin's theorem 
    \begin{eqnarray}
    |\lambda-{J}_{ii}|\leq \sum_{k\neq i} |{J}_{ki}|.
    \end{eqnarray}
    Since $\boldsymbol{\Omega}=\boldsymbol{\Omega}_1$ and $\vec b=\vec 1$, we have
    \[
    {J}_{ii}=\frac{1}{N} f'(x^*)-\frac{N-1}{N} D_{ii} \qquad \sum_{k\neq i} |{J}_{ki}|=\frac{N-1}{N} |f'(x^*)+D_{ii}|=R_i
    \]
    Let $\bar D=\text{max}_i |D_{ii}|$ and $\bar R=\frac{N-1}{N}|f'(x^*)|$.  
    It follows that the eigenvalues $\lambda\in \Lambda(\boldsymbol{J})$ must be enclosed, in the complex plane, in circles of radius $\bar R$ and center $z_i=\frac{1}{N} f'(x^*)-\frac{N-1}{N} D_{ii}$.
\end{itemize}

\begin{figure}
    \centering
    \includegraphics[scale=0.5]{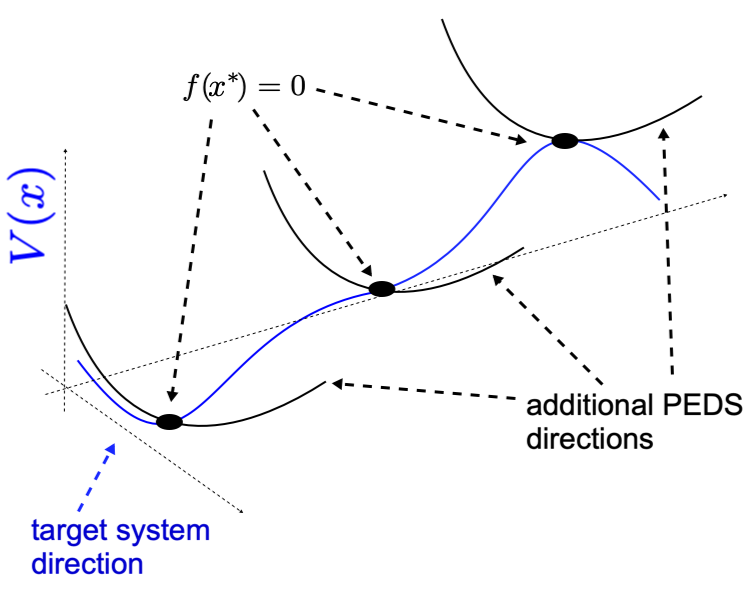}
    \caption{Interpreting the $N$ dimensional PEDS dynamical system fixed points, from the point of view of the potential in Fig.~\ref{fig:onedfig} (blue curve). Notice that while the PEDS procedure preserves the extrema of potential $V(x)$, the embedded dynamical system cannot, in general, be expressed as the gradient of a generalized potential.}
    \label{fig:onedfig2}
\end{figure}

\subsubsection{General case: Vector target system}
The analysis of the PEDS embedding for scalar target systems carried out in the previous section showed that exploiting the uniform mean field projector and the standard or generalized decay functions, the embedding Jacobian can be fully characterized on the basis of the features of the target system fixed points (stable, unstable, or neutral). We have obtained this result both for the standard commutative map, in which essentially one ``mixes" linearly the dynamical systems, and in the case of the standard non-commuting map, in which the mix is non-trivial. The difference between the two is, thus, essentially contained in the embedding intermediate dynamics.

While in higher dimensions the nature of the questions to be answered is quite similar, the derivations are technically more challenging. The reason lies in the ordering, that, as mentioned earlier, does play a role in how the PEDS is defined. However, the case $\boldsymbol{\Omega}=\boldsymbol{\Omega}_1$ still makes it possible to carry out an almost entirely analytical derivation, even if at least some results can be proved for a more general projector structure. 

Let us focus on the following PEDS for an $m$-dimensional target system $$\mathcal O=\{\{f_1(\vec x),\cdots,f_m(\vec x)\},\boldsymbol{\Omega}_1,-\{\boldsymbol{Q}_1(\boldsymbol{\Omega}_1)\vec X,\cdots,\boldsymbol{Q}_m(\boldsymbol{\Omega}_1)\vec X\},\vec 1,N\}.$$
Therefore we consider an extended system as in \eqref{eq:ext} characterized by the mean field projector and any eligible decay functions.

Similarly to the scalar case, we also consider here the two generalizations $\boldsymbol{Q}$ to the standard decay functions defined in \eqref{eq:genab1}. The corresponding PEDS take the form $\mathcal O_A=(\{f_i(x)\},\boldsymbol{\Omega}_1,\{-\boldsymbol{D}(\boldsymbol{I}-\boldsymbol{\Omega}_1)\vec X_i\},S,\vec 1,N)$ and $\mathcal O_B=(\{f_i(x)\},\boldsymbol{\Omega}_1,\{-(\boldsymbol{I}-\boldsymbol{\Omega}_1)\boldsymbol{D}(\boldsymbol{I}-\boldsymbol{\Omega}_1)\vec X_i\},S,\vec 1,N)$, which in terms of PEDS equations, read
\begin{equation}
    \frac{d\vec X_i}{dt}=\boldsymbol{\Omega}_1 \boldsymbol{F}_i(\vec X_1,\dots,\vec X_m)\vec 1 -\boldsymbol{Q}_{i}(\boldsymbol{\Omega}_1)\vec X_i 
    \label{eq:natale}
\end{equation}
where
\begin{equation}
    \boldsymbol{Q}_{i}(\boldsymbol{\Omega}_1)=\begin{cases}
    \boldsymbol{D}_i(\boldsymbol{I}-\boldsymbol{\Omega}_1) & \text{generalization A} \\[1ex]
    (\boldsymbol{I}-\boldsymbol{\Omega}_1)\boldsymbol{D}_i(\boldsymbol{I}-\boldsymbol{\Omega}_1) & \text{generalization B}
    \end{cases}
\end{equation}
being $\boldsymbol{D}_i$ positive, diagonal matrices that make the corresponding decay function $\boldsymbol{\Omega}_1$-eligible.

For any eligible decay function, we can obtain the vector target system Jacobian for the PEDS system as follows.
Let us consider first the standard commuting map. We use the representation in \eqref{eq:diagoform}, so that
\begin{eqnarray}
    F^{(c)}_{i,k}(\vec X_1,\dots,\vec X_m)=f_i(X_{1,k},\dots,X_{m,k})b_k=f_i(X_{1,k},\dots,X_{m,k})
\end{eqnarray}
as $\vec b=\vec 1$. We evaluate the Jacobian in blocks, starting from component 1 as in the scalar case
\begin{eqnarray}
    (\boldsymbol{J}_{ij}^{(c,1)}(\{\vec X_i\}))_{ks}=\sum_{r=1}^N \boldsymbol{\Omega}_{1,kr} \frac{\partial f_i(X_{1,r},\dots,X_{m,r})}{\partial X_{j,s}}=\boldsymbol{\Omega}_{1,ks} \frac{\partial f_i(X_{1,s},\cdots,X_{m,s})}{\partial X_{j,s}}
\end{eqnarray}
thus, in the fixed point $\vec X_i^*=x_i^*\vec 1$ (see the multivariate banality of the mean Corollary~\ref{coro:multibanality}) we find
\begin{eqnarray}
\boldsymbol{J}_{ij}^{(c)}(\{\vec X_i^*\})=\boldsymbol{\Omega}_1 f'_{i,x_j}(\vec x^*)+\delta_{ij} \boldsymbol{Q}_i(\boldsymbol{\Omega}_1).
\end{eqnarray}
where $f'_{i,x_j}=\partial f_i(\vec x)/\partial x_j$. In other words, the Jacobian in the equilibria can be built as
\begin{eqnarray}
     \begin{pmatrix}
     f'_{1,x_1}(\vec x^*) \boldsymbol{\Omega}_1+\boldsymbol{Q}_1(\boldsymbol{\Omega}_1) & f'_{1,x_2}(\vec x^*) \boldsymbol{\Omega}_1 
     & \cdots & f'_{1,x_n}(\vec x^*) \boldsymbol{\Omega}_1\\
     f'_{2,x_1}(\vec x^*) \boldsymbol{\Omega}_1 & f'_{2,x_2}(\vec x^*) \boldsymbol{\Omega}_1+\boldsymbol{Q}_2(\boldsymbol{\Omega}_1) 
     &  \cdots &\vdots \\
      \vdots & \vdots 
      & \vdots & \vdots \\
      f'_{m,x_1}(\vec x^*) \boldsymbol{\Omega}_1 & \cdots 
      & \cdots & f'_{m,x_m}(\vec x^*) \boldsymbol{\Omega}_1+\boldsymbol{Q}_m(\boldsymbol{\Omega}_1)
     \end{pmatrix}
     \label{eq:myfirstJacobian}
 \end{eqnarray}
 where each block is of size $N\times N$.

Surprisingly, the same result holds also for the mixed and standard non-com\-mu\-ta\-tive maps. We start by proving this in the mixed commuting case, where ordering is immaterial.
\begin{proposition}\label{prop:generale}
Let $\mathcal O=(\{ f_1(\vec x),\cdots,f_m(\vec x)\},\boldsymbol{\Omega}_1,\{\boldsymbol{Q}_i(\boldsymbol{\Omega}_1) \vec X_i\},S,\vec 1,N)$ be the PEDS built on a mixed commutative map, considering the decay functions $\boldsymbol{Q}_i(\boldsymbol{\Omega}_1) \vec X$ assumed to be  $\boldsymbol{\Omega}_1$-eligible. Then, for any ordering $S$, the  Jacobian matrix, evaluated at the equilibrium $\vec X_i^*=x_i^* \vec 1$ being $\vec x$ an equilibrium of the target system, is given by
\begin{eqnarray}
\boldsymbol{J}_{ij}^{(mc)}(\{\vec X_i^*\})=\boldsymbol{\Omega}_1 f'_{i,x_j}(\vec x^*)+\delta_{ij} \boldsymbol{Q}_i(\boldsymbol{\Omega}_1).
\end{eqnarray}
\end{proposition}
\begin{proof}
We aim at evaluating the Jacobian of\footnote{As usual, we use a general projection operator wherever possible.}
\begin{eqnarray}
    \boldsymbol{F}_i^{(mc)}(\boldsymbol{X}_1,\cdots,\boldsymbol{X}_m)=a_{i,0}\boldsymbol{I}+\sum_{k=1}^\infty \sum_{j_1,\cdots,j_m}^k a_{k,i;j_1,\cdots,j_m} (\boldsymbol{\Omega}(\boldsymbol{X}_1^{j_1}\cdots \boldsymbol{X}_m^{j_m})^{1/k})^k.
\end{eqnarray}
as it is needed for the estimation of the first Jacobian component
\begin{eqnarray}
    (\boldsymbol{J}^{(mc,1)}_{ij}(\{\vec X_i\}))_{st}&=&\frac{\partial(\boldsymbol{F}_i^{(mc)}(\boldsymbol{X}_1,\cdots,\boldsymbol{X}_m)\vec 1)_s}{\partial X_{j,t}}\nonumber \\
    &=&\sum_{k=1}^\infty \sum_{j_1,\cdots,j_m}^k a_{k,i;j_1,\cdots,j_m} \left( \frac{\partial (\boldsymbol{\Omega}(\boldsymbol{X}_1^{j_1}\cdots \boldsymbol{X}_m^{j_m})^{1/k})^k}{\partial X_{j,t}} \vec 1\right)_s.
    \label{eq:angelo}
\end{eqnarray}
Therefore, we need to evaluate 
\begin{eqnarray}
    \frac{\partial (\boldsymbol{\Omega}(\boldsymbol{X}_1^{j_1}\cdots \boldsymbol{X}_m^{j_m})^{1/k})^k}{\partial X_{j,t}}
\end{eqnarray}
We exploit the identity
\begin{eqnarray}
    \frac{\partial \boldsymbol{B}^k}{\partial X}=\sum_{z=0}^{k-1} \boldsymbol{B}^{k-1-z} \frac{\partial \boldsymbol{B}}{\partial X}\boldsymbol{B}^z \qquad k\ge 1
    \label{eq:identity}
\end{eqnarray}
valid for any matrix $\boldsymbol{B}$, obtaining
\begin{eqnarray}
\frac{\partial (\boldsymbol{\Omega}(\boldsymbol{X}_1^{j_1}\cdots \boldsymbol{X}_m^{j_m})^{1/k})^k}{\partial X_{j,t}}
&=&\sum_{z=0}^{k-1} \Big(\boldsymbol{\Omega}(\boldsymbol{X}_1^{j_1}\cdots \boldsymbol{X}_m^{j_m})^{1/k}\Big)^{k-1-z} \left(\boldsymbol{\Omega} 
\frac{\partial \big((\boldsymbol{X}_1^{j_1}\cdots \boldsymbol{X}_m^{j_m})^{1/k}\big)}{\partial X_{j,t}}\right)
 \nonumber \\
    & & \ \ \ \  \ \times \Big(\boldsymbol{\Omega}(\boldsymbol{X}_1^{j_1}\cdots \boldsymbol{X}_m^{j_m})^{1/k} \Big)^z
    \label{eq:branduardi}
\end{eqnarray}
where, since $\boldsymbol{X}_i$ are diagonal, we have
$$
\left(\frac{\partial \big((\boldsymbol{X}_1^{j_1}\cdots \boldsymbol{X}_m^{j_m})^{1/k}\big)}{\partial X_{j,t}}\right)_{pq}=
\delta_{pq}\delta_{pt} \left( \prod_{r\neq j}X_{r,t}^{j_t/k}  \right)\frac{j_j}{k} X_{j,t}^{j_j/k-1}.$$
Now, taking into account that $\boldsymbol{\Omega}=\boldsymbol{\Omega}_1$ and that the fixed points are defined by $\vec X_i^*=x_i^*\vec 1$ (because of the multivariate banality of the mean Corollary~\ref{coro:multibanality}), we find
\begin{eqnarray}
    \left.\big((\boldsymbol{\Omega_1}(\boldsymbol{X}_1^{\alpha_1}\cdots \boldsymbol{X}_m^{\alpha_m})^{1/\eta})^\beta\big)_{ij} \right|_{\{\vec X_i^*\}}&=& {\Omega}_{1,ij} \prod_{i=1}^m (x_i^*)^{\frac{\beta}{\eta} \alpha_i}\\
    \left.\delta_{pq}\delta_{pt} \left( \prod_{r\neq j}X_{r,t}^{j_t/k}  \right)\frac{j_j}{k} X_{j,t}^{j_j/k-1}\right|_{\{\vec X_i^*\}}&=&\delta_{pq}\delta_{pt} \left( \prod_{r\neq j}(x_{r}^*)^{j_t/k}  \right)\frac{j_j}{k} (x_{j}^*)^{j_j/k-1}
\end{eqnarray}
thus, substituting into \eqref{eq:branduardi}
\begin{align}
    &\left.\left(\frac{\partial (\boldsymbol{\Omega}_1(\boldsymbol{X}_1^{j_1}\cdots \boldsymbol{X}_m^{j_m})^{1/k})^k}{\partial X_{j,t}}\right)_{pq}\right|_{\{\vec X_i^*\}} =\sum_{b,c,d=1}^N \sum_{z=0}^{k-1} \left.\big((\boldsymbol{\Omega_1}(\boldsymbol{X}_1^{j_1}\cdots \boldsymbol{X}_m^{j_m})^{1/k})^{k-1-z}\big)_{pb} \right|_{\{\vec X_i^*\}} \nonumber\\[1ex]
    &\quad \times \left.\left(\Omega_{1,bc}
    \frac{\partial \left((\boldsymbol{\Omega}_1(\boldsymbol{X}_1^{j_1}\cdots \boldsymbol{X}_m^{j_m})^{1/k}\right)_{cd}}{\partial X_{j,t}}\right|_{\{\vec X_i^*\}}\right) \left.\big((\boldsymbol{\Omega_1}(\boldsymbol{X}_1^{j_1}\cdots \boldsymbol{X}_m^{j_m})^{1/k})^{z}\big)_{dq} \right|_{\{\vec X_i^*\}}\nonumber\\[1ex]
    &\quad = j_j(x_j^*)^{j_j-1}\prod_{r\neq j} (x_r^*)^{j_t}\sum_{b,c,d=1}^N {\Omega}_{1,pb} {\Omega}_{1,bc} \delta_{cd}\delta_{ct} {\Omega}_{1,dq}\nonumber\\[1ex]
    &\quad = j_j(x_j^*)^{j_j-1}\prod_{r\neq j} (x_r^*)^{j_t}\Omega_{1,pt}\Omega_{1,tq}
\end{align}
where each element of $\boldsymbol{\Omega}_1$ is equal to $1/N$. Substituting into \eqref{eq:angelo} we recognize the series expansion of $f'_{i,x_j}(\vec x^*)$. This leads to the final expression
\begin{equation}
    (\boldsymbol{J}^{(mc,1)}_{ij}(\{\vec X_i^*\}))_{st} = f'_{i,x_j}(\vec x^*) \Omega_{1,st}.
\end{equation}
Taking into account the second Jacobian component, i.e. the Jacobian of the decay functions, finally yields the result to be proved
\begin{equation}
    \boldsymbol{J}^{(mc)}_{ij}(\{\vec X_i^*\}) = f'_{i,x_j}(\vec x^*) \boldsymbol{\Omega}_{1}+\delta_{ij}\boldsymbol{Q}_i(\boldsymbol{\Omega}_1).
\end{equation}
\end{proof}

Finally, let us turn to the non-commutative map Jacobian, for which we have the following result.
\begin{proposition}
Let $\mathcal O=(\{ f_1(\vec x),\cdots,f_m(\vec x)\},\boldsymbol{\Omega}_1,\{\boldsymbol{Q}_i(\boldsymbol{\Omega}_1) \vec X_i\},S,\vec 1,N)$ be the PEDS built on a non-commutative map, considering the decay functions $\boldsymbol{Q}_i(\boldsymbol{\Omega}_1) \vec X$ assumed to be  $\boldsymbol{\Omega}_1$-eligible. Then, for any ordering $S$, the Jacobian matrix, evaluated at the equilibrium $\vec X_i^*=x_i^* \vec 1$, is given by
\begin{eqnarray}
\boldsymbol{J}_{ij}^{(nc)}(\{\vec X_i^*\})=\boldsymbol{\Omega}_1 f'_{i,x_j}(\vec x^*)+\delta_{ij} \boldsymbol{Q}_i(\boldsymbol{\Omega}_1).
\end{eqnarray}
\end{proposition}
\begin{proof}
As we aim at demonstrating the independence of the Jacobian from the ordering $S$ (at the fixed points), we consider the formalism introduced in Section~\ref{sec:ordering}. Given the matrix map definition \eqref{eq:vectaylorexp}, we use the general form
\begin{equation}
    \boldsymbol{F}_s(\vec X)=\sum_{k=0}^\infty \sum_{i_1,\cdots,i_j=0 }^ja_{s,k; i_1,\cdots, i_m} \{(\boldsymbol{\Omega} \boldsymbol{X}_1)^{i_1}\cdots (\boldsymbol{\Omega}  \boldsymbol{X}_m)^{i_m} \}_S
\end{equation}
that forms the basis for the evaluation of the first Jacobian component. Writing, to simplify the notation, $o_{\sigma(s_1) \cdots \sigma(s_m)}=o_{\vec \sigma}$, we compute the derivatives
\begin{align}
    &\frac{\partial \phantom{X_{a,b}}}{\partial X_{a,b}} \{(\boldsymbol{\Omega}\boldsymbol{X}_1)^{i_1}\cdots (\boldsymbol{\Omega}  \boldsymbol{X}_m)^{i_m} \}_S \nonumber\\
    &\qquad =\sum_{\sigma\in \mathcal S(m)} o_{\vec \sigma} \frac{\partial \phantom{X_{a,b}}}{\partial X_{a,b}}
    \left( (\boldsymbol{\Omega}  \boldsymbol{X}_{\sigma(s_1)})^{i_1}\cdots (\boldsymbol{\Omega}  \boldsymbol{X}_{\sigma(s_m)})^{i_m} \right) \nonumber\\
    &\qquad =\sum_{\sigma\in \mathcal S(m)} o_{\vec \sigma} \left( \frac{\partial \phantom{X_{a,b}}}{\partial X_{a,b}} (\boldsymbol{\Omega}  \boldsymbol{X}_{\sigma(s_1)})^{i_1} \right) \cdots (\boldsymbol{\Omega}  \boldsymbol{X}_{\sigma(s_m)})^{i_m}+\cdots \nonumber\\
    &\qquad\quad + \sum_{\sigma\in \mathcal S(m)} o_{\vec \sigma} (\boldsymbol{\Omega}  \boldsymbol{X}_{\sigma(s_1)})^{i_1}\cdots \left(\frac{\partial \phantom{X_{a,b}}}{\partial X_{a,b}} (\boldsymbol{\Omega}  \boldsymbol{X}_{\sigma(s_m)})^{i_m}\right) \nonumber\\
     &\qquad = \sum_{\sigma\in \mathcal S(m)} o_{\vec \sigma}\delta_{a,\sigma(s_1)} \left( \frac{\partial \phantom{X_{\sigma(s_1),b}}}{\partial X_{\sigma(s_1),b}} (\boldsymbol{\Omega}  \boldsymbol{X}_{\sigma(s_1)})^{i_1} \right) \cdots (\boldsymbol{\Omega}  \boldsymbol{X}_{\sigma(s_m)})^{i_m}+\cdots \nonumber\\
     &\qquad\quad + \sum_{\sigma\in \mathcal S(m)} o_{\vec \sigma} \delta_{a,\sigma(s_m)} (\boldsymbol{\Omega}  \boldsymbol{X}_{\sigma(s_1)})^{i_1}\cdots \left(\frac{\partial \phantom{X_{\sigma(s_m),b}}}{\partial X_{\sigma(s_m),b}}  (\boldsymbol{\Omega}  \boldsymbol{X}_{\sigma(s_m)})^{i_m}\right)
     \label{eq:fixedpoint}
\end{align}

Applying identity \eqref{eq:identity} to matrix $(\boldsymbol{\Omega}\boldsymbol{X}_{s_k})^{i_k}$ we find
\begin{equation}
  \frac{\partial \phantom{X_{s_k,b}}}{\partial X_{s_k,b}} \left( (\boldsymbol{\Omega}\boldsymbol{X}_{s_k})^{i_k} \right)_{ij} =  \sum_{t=0}^{i_k-1} \sum_{l=1}^N \left((\boldsymbol{\Omega} \boldsymbol{X}_{s_k})^t\right)_{il} {\Omega}_{lb} \left((\boldsymbol{\Omega} \boldsymbol{X}_{s_k})^{i_k-1-t}\right)_{bj},
\end{equation}

At this point, the multivariate banality of the mean Corollary~\ref{coro:multibanality} proves that the PEDS fixed points are given by $\boldsymbol{X}_{s_k}^*=x_{s_k}^* \boldsymbol{I}$, therefore for the mean field  projector $\boldsymbol{\Omega}=\boldsymbol{\Omega}_1$ we find
\begin{align}
    \left.\frac{\partial \phantom{X_{s_k,b}}}{\partial X_{s_k,b}} \left( (\boldsymbol{\Omega}_1\boldsymbol{X}_{s_k})^{i_k} \right)_{ij}\right|_{\{\vec X_i^*\}} &= \sum_{t=0}^{i_k-1} (x^{*}_{s_k})^{i_k-1}\sum_{s=1}^N \Omega_{1,is}^t \Omega_{1,sb} \Omega_{1,bj}^{i_k-1-t} \nonumber\\
    &= i_k(x^{*}_{s_k})^{i_k-1} \Omega_{1,ib}\Omega_{1,bj}
\end{align}

Substituting this expression into the derivative of the ordered product, yields
\begin{align}
    &\left. \frac{\partial \phantom{X_{a,b}}}{\partial X_{a,b}} \{(\boldsymbol{\Omega}_1\boldsymbol{X}_1)^{i_1}\cdots (\boldsymbol{\Omega}_1  \boldsymbol{X}_m)^{i_m} \}_S \right|_{\{\vec X_i^*\}}  \nonumber\\
    &\qquad =\sum_{\sigma\in \mathcal S(m)} o_{\vec \sigma}\delta_{a,\sigma(s_1)}\Big(i_{1} (x^{*}_{\sigma(s_1)})^{i_1-1} (\boldsymbol{\Omega}_1)_{:b}(\boldsymbol{\Omega}_1)_{b:}\Big)\nonumber\\
    &\qquad\qquad\times (x^{*}_{\sigma(s_{2})})^{i_2}\boldsymbol{\Omega}_1^{i_2}\cdots (x^{*}_{\sigma(s_m)})^{i_m}\boldsymbol{\Omega}_1^{i_m} + \cdots \nonumber \\
    &\qquad\quad+\sum_{\sigma\in \mathcal S(m)} o_{\vec \sigma}\delta_{a,\sigma(s_m)} (x^{*}_{\sigma(s_{1})})^{i_1}\boldsymbol{\Omega}_1^{i_1}(x^{*}_{\sigma(s_{2})})^{i_2}\boldsymbol{\Omega}_1^{i_2}\cdots
    \nonumber\\
    &\qquad\qquad\times (x^{*}_{\sigma(s_{m-1})})^{i_{m-1}}\boldsymbol{\Omega}_1^{i_{m-1}} \Big(i_{m} (x^{*}_{\sigma(s_m)})^{i_m-1} (\boldsymbol{\Omega}_1)_{:b}(\boldsymbol{\Omega}_1)_{b:}\Big)
    \nonumber\\
    &\qquad =\sum_{\sigma\in \mathcal S(m)} o_{\vec \sigma}\delta_{a,\sigma(s_1)}i_1(x^{*}_{\sigma(s_{1})})^{i_1-1}(x^{*}_{\sigma(s_{2})})^{i_2}\cdots (x^{*}_{\sigma(s_{m})})^{i_m}(\boldsymbol{\Omega}_1)_{:b}(\boldsymbol{\Omega}_1)_{b:} + \cdots \nonumber\\
    &\qquad\quad+\sum_{\sigma\in \mathcal S(m)} o_{\vec \sigma}\delta_{a,\sigma(s_m)}i_m (x^{*}_{\sigma(s_{1})})^{i_1}(x^{*}_{\sigma(s_{2})})^{i_2}\cdots (x^{*}_{\sigma(s_{m})})^{i_m-1}(\boldsymbol{\Omega}_1)_{:b}(\boldsymbol{\Omega}_1)_{b:}.
    \label{eq:fixedpointgen2}
\end{align}
where $(\boldsymbol{\Omega}_1)_{:b}$ and $(\boldsymbol{\Omega}_1)_{b:}$ are vectors made of the $b$-th column and row  of $\boldsymbol{\Omega}_1$, respectively. As the matrix products all collapse into the same quantity, the previous expression is independent of the ordering $S$. Therefore, we can write
\begin{equation}
    \left. \frac{\partial \phantom{X_{a,b}}}{\partial X_{a,b}} \{(\boldsymbol{\Omega}_1\boldsymbol{X}_1)^{i_1}\cdots (\boldsymbol{\Omega}_1  \boldsymbol{X}_m)^{i_m} \}_S \right|_{\{\vec X_i^*\}}=\left(\frac{\partial\phantom{x_a}}{\partial x_a}\prod_{j=1}^m (x_j^*)^{i_j}\right) (\boldsymbol{\Omega}_1)_{:b}(\boldsymbol{\Omega}_1)_{b:}.
\end{equation}
This means that, for any ordering $S$, at the fixed point the sum of the terms for the derivative with respect to the elements of each extended variable $\vec X_a$ leads, once taking into account the factor $\vec 1$ in \eqref{eq:natale}, to a scalar factor corresponding to $f'_{i,x_a}(\vec x^*)$, i.e. the corresponding element of the Jacobian of  the target system multiplied times matrix $\boldsymbol{\Omega}_1$.

Concerning the second part of the Jacobian, i.e. the derivatives of $\boldsymbol{Q}_i(\boldsymbol{\Omega}_1) \vec X_i$, the result is a block diagonal matrix of the type $\text{diag}\{\boldsymbol{Q}_i(\boldsymbol{\Omega}_1)\}$.

In summary, even in this case the full Jacobian at the fixed points follows the block structure claimed in the proposition.
\end{proof} 
We are now ready to discuss the Jacobian spectral properties irrespective of the chosen map, as the matrix is the same for all of the three maps that we consider. For the sake of simplicity, we limit the discussion to the standard decay functions $\boldsymbol{Q}_i(\boldsymbol{\Omega}_1)=-\alpha_i (\boldsymbol{I}-\boldsymbol{\Omega}_1)$, so that
\begin{align}
     \boldsymbol{J}(\{\vec X_i^*\})&=\begin{pmatrix}
     f'_{1,x_1}(\vec x^*) \boldsymbol{\Omega}_1 & f'_{1,x_2}(\vec x^*) \boldsymbol{\Omega}_1 & \cdots & \cdots & f'_{1,x_n}(\vec x^*) \boldsymbol{\Omega}_1\\
     f'_{2,x_1}(\vec x^*) \boldsymbol{\Omega}_1 & f'_{2,x_2}(\vec x^*) \boldsymbol{\Omega}_1 & f'_{2,x_3}(\vec x^*) \boldsymbol{\Omega}_1 & \cdots & f'_{2,x_n}(\vec x^*) \boldsymbol{\Omega}_1\\
     \vdots & \vdots & \ddots & \vdots & \vdots \\
      f'_{n,x_1}(\vec x^*) \boldsymbol{\Omega}_1 & \cdots & \cdots & \cdots & f'_{n,x_n}(\vec x^*) \boldsymbol{\Omega}_1
     \end{pmatrix} \nonumber \\
     &-\begin{pmatrix}
     \alpha_1 (\boldsymbol{I}-\boldsymbol{\Omega}_1) & 0 & \cdots & \cdots & 0 \\
     0 & \alpha_2 (\boldsymbol{I}-\boldsymbol{\Omega}_1) &  &  \cdots &\vdots \\
      \vdots & \vdots & \ddots & \vdots & 0 \\
      0 & \cdots & \cdots &  0 &\alpha_m (\boldsymbol{I}-\boldsymbol{\Omega}_1)
     \end{pmatrix}
     \label{eq:myfirstJacobian2}
 \end{align}
that can be cast in the following form
\begin{equation}
    \boldsymbol{J}(\{\vec X_i^*\})= (\boldsymbol{J}_m(\vec x^*)+\boldsymbol{D}^1_\alpha)\otimes \boldsymbol{\Omega}_1-\boldsymbol{D}^N_\alpha
    \label{eq:myfirstJacobian3}
\end{equation}
where $\boldsymbol{J}_m(\vec x^*)$ is the Jacobian of the target system functions $\{f_i(\vec x)\}$ evaluated at the target system equilibrium $\vec x^*$, $\otimes$ denotes matrix Kronecker product\footnote{According to the definition, the $i,j$ block of the matrix Kronecker product $A\otimes B$ is $a_{ij} B$.}, and
\[
\boldsymbol{D}^k_\alpha=\text{diag}(\underbrace{\alpha_1,\cdots,\alpha_1}_{\text{$k$ times}},\underbrace{\alpha_2,\cdots,\alpha_2}_{\text{$k$ times}},\cdots,\underbrace{\alpha_m,\cdots,\alpha_m}_{\text{$k$ times}}).
\]

The Jacobian \eqref{eq:myfirstJacobian3} is a generalization of the scalar case (see Proposition~\ref{prop:onepedsjac}). We are interested in assessing the properties of its eigenvalues.  

For the time being, we discuss the simpler case $\alpha_i\equiv \alpha$, and since it will be useful later, let us think of this Jacobian for a general $\boldsymbol{\Omega}$, only to then consider $\boldsymbol{\Omega_1}$ as a special case.\footnote{To motivate this generalized discussion, we briefly anticipate the result of an upcoming paper, in which we show that \eqref{eq:jacgen} is in fact the first term of the representation obtained for the Jacobian of a general projector $\boldsymbol{\Omega}$. This general case is, however, beyond the scope of this paper.}  Let us therefore discuss the spectrum of 
\begin{eqnarray}
    \boldsymbol{J}(\{\vec X_i^*\})=(\boldsymbol{J}_m(\vec x^*)+\alpha \boldsymbol{I}_m)\otimes \boldsymbol{\Omega}-\alpha \boldsymbol{I}_{Nm}.
    \label{eq:jacgen}
\end{eqnarray}
where $\boldsymbol{I}_q$ is the identity matrix of size $q\times q$.
Being $\lambda_i$ the eigenvalues of $\boldsymbol{J}_m(\vec x^*)$, the eigenvalues of $\boldsymbol{J}_{m}(\vec x^*)+\alpha \boldsymbol{I}_m$ are given by $\lambda_i+\alpha$, and $-\alpha$ are the eigenvalues of $-\alpha \boldsymbol{I}_{Nm}$ for the whole matrix. Let us assume that $\boldsymbol{\Omega}$ has $k$ unitary eigenvalues ($k=1$ for $\boldsymbol{\Omega}=\boldsymbol{\Omega}_1$), while the remaining $N-k$ eigenvalues are equal to $0$. 
Then  matrix $\boldsymbol{J}(\{\vec X_i^*\})$ has eigenvalues \cite{Bernstein}
\begin{itemize}
    \item $\lambda_i$, $1\leq i\leq m$ with multiplicity $k$
    \item $-\alpha$ with multiplicity $m(N-k)$.
\end{itemize}
As a consequence, if $\lambda_i<0$ $\forall i$, then a stable equilibrium point for the target system is still stable in the PEDS embedding. Similarly, if the equilibrium  point is unstable, or if at least some $i$ values exist for which  $\lambda_i>0$, then it becomes a saddle point for the extended system, being characterized by $m(N-k)$ negative eigenvalues and $mk$ positive eigenvalues. 
Thus, the following classification holds
\begin{eqnarray}
    \{\vec X^*_i\}=\begin{cases}
    \text{stable} & \text{if $\vec x^*$ is stable},\\
    \text{saddle point} & \text{if $\vec x^*$ is a saddle point}, \\
    \text{saddle point} & \text{if $\vec x^*$ is unstable}.
    \end{cases}
    \label{eq:2}
\end{eqnarray}

This analysis suggests that the presence of ``barriers'' in the target system, characterized by unstable equilibria, can (in principle) be overcome in the PEDS embedding via their transformation into saddle points in the extended system.

\subsection{Dynamical ordering-equivalence for the uniform mean field projector}

The uniform mean field projector has various properties that are interesting \textit{per se}. In particular, we wish to show here that not only the fixed points, but also the embedding dynamics is ordering independent.

For this purpose, consider a PEDS of the form $$\mathcal O_r=(\{f_i(x_1,\cdots,x_m)\},\boldsymbol{\Omega_1},\{\boldsymbol{G}_i(\boldsymbol{\Omega}_1)\vec X_i\},S_r, \vec 1,N),$$ and for arbitrary $\boldsymbol{\Omega}_1$-eligible decay functions. Given the PEDS above, the standard non-commutative matrix embedding is given by
\begin{eqnarray}
   \vec F_s= \sum_{k=0}^\infty \sum_{i_1,\cdots,i_j }^ka_{s,k; i_1\cdots i_m} \{(\boldsymbol{\Omega}_1  \boldsymbol{X}_1)^{i_1}\cdots (\boldsymbol{\Omega}_1  \boldsymbol{X}_m)^{i_m} \}_{S_r}\vec 1
\end{eqnarray}

We prove the following
\begin{proposition}\label{prop:independence}
The quantity 
$$\{(\boldsymbol{\Omega}_1  \boldsymbol{X}_1)^{i_1}\cdots (\boldsymbol{\Omega}_1  \boldsymbol{X}_m)^{i_m} \}_{S_r}\vec 1$$
is  independent of the ordering $S_r$, for any $i_1,\cdots, i_m$.
\end{proposition}
\begin{proof}
The proof relies on the following observation. We have in general that
\begin{eqnarray}
    \boldsymbol{\Omega}_1  \boldsymbol{X}_s \boldsymbol{\Omega}_1  =\langle X_s\rangle \boldsymbol{\Omega}_1.
    \label{eq:trick}
\end{eqnarray}
with $\langle X_s\rangle=\frac{1}{N} \sum_{j=1}^N X_s^j$. The previous result can be easily shown as follows
\begin{eqnarray}
    (\boldsymbol{\Omega}_1  \boldsymbol{X}_s \boldsymbol{\Omega}_1)_{ij}&=&\sum_{k_1,k_2=1}^N \Omega_{1,ik_1}  {X}_{s,k_1 k_2} {\Omega}_{1,k_2 j}\nonumber \\
    &=&\frac{1}{N^2} \sum_{k_1,k_2=1}^N X_{s,k_1} \delta_{k_1 k_2}\nonumber \\
    &=&\frac{1}{N} \sum_{k_1=1}^N X_{s,k_1} \frac{1}{N}=\langle X_s\rangle  {\Omega}_{1,ij}
\end{eqnarray}
Because of \eqref{eq:trick}, we can always write
the following
\begin{eqnarray}
    (\boldsymbol{\Omega}_1  \boldsymbol{X}_{\sigma(1)})^{\sigma(i_1)}\cdots (\boldsymbol{\Omega}_1  \boldsymbol{X}_{\sigma(m)})^{\sigma(i_m)}=f_{\sigma(i_1),\cdots,\sigma(i_{m})}(\langle{X}_{\sigma(1)}\rangle,\cdots, \langle{X}_{\sigma(m)}\rangle) \boldsymbol{\Omega}_1  \boldsymbol{X}_{\sigma(m)}.
\end{eqnarray}
where function $f$ is scalar. To gain an intuition about the scalar $f$, consider for instance $(\boldsymbol{\Omega}_1  \boldsymbol{X}_{1})^a(\boldsymbol{\Omega}_1  \boldsymbol{X}_{2})^b$.
Using \eqref{eq:trick}, the previous expression  can be written as 
\begin{eqnarray}
    (\boldsymbol{\Omega}_1  \boldsymbol{X}_{1})^a(\boldsymbol{\Omega}_1  \boldsymbol{X}_{2})^b=\langle X_1\rangle^{a-1} \langle X_2\rangle^{b-1}\boldsymbol{\Omega}_1  \boldsymbol{X}_{1}\boldsymbol{\Omega}_1  \boldsymbol{X}_{2}=\langle X_1\rangle^{a} \langle X_2\rangle^{b-1}\boldsymbol{\Omega}_1  \boldsymbol{X}_{2},
\end{eqnarray}
so that $f_{12}=\langle X_1\rangle^{a} \langle X_2\rangle^{b-1}$. 
At this point we have
\begin{eqnarray}
\{(\boldsymbol{\Omega}_1  \boldsymbol{X}_1)^{i_1}\cdots (\boldsymbol{\Omega}_1  \boldsymbol{X}_m)^{i_m} \}_{S_r}\vec 1 &=&\sum_{\sigma\in S_m} o_{\vec \sigma}f_{\sigma(i_1),\cdots,\sigma(i_{m})}(\langle{X}_{\sigma(1)}\rangle,\cdots, \langle{X}_{\sigma(m)}\rangle) \boldsymbol{\Omega}_1  \boldsymbol{X}_{\sigma(m)}\vec 1\nonumber \\
&=&\sum_{\sigma\in S_m} o_{\vec \sigma}f_{\sigma(i_1),\cdots,\sigma(i_{m})}(\langle{X}_{\sigma(1)}\rangle,\cdots, \langle{X}_{\sigma(m)}\rangle) \langle X_{\sigma(m)}\rangle\vec 1\nonumber\nonumber \\
&=&\sum_{\sigma\in S_m} o_{\vec \sigma}f_{\sigma(i_1),\cdots,\sigma(i_{m})}(\langle{X}_{\sigma(1)}\rangle,\cdots, \langle{X}_{\sigma(m)}\rangle) \langle X_{\sigma(m)}\rangle\vec 1\nonumber\nonumber \\
&=&f_{i_1,\cdots,i_{m}}(\langle{X}_{\sigma(1)}\rangle,\cdots, \langle{X}_{\sigma(m)}\rangle) \langle X_{m}\rangle\vec 1
\end{eqnarray}
which follows from the fact that the scalar variables $\langle X_j\rangle$ do commute.
\end{proof}
Proposition \ref{prop:independence} is important because it implies the following
\begin{corollary} \label{cor:independence} \textbf{Dynamical ordering independence for the uniform mean field projector}. 
For any analytic functions $f_i$, we have 
\begin{eqnarray}
\mathcal O_r&=&(\{f_i(x_1,\cdots,x_m)\},\boldsymbol{\Omega_1},\{\boldsymbol{Q_i}(\boldsymbol{\Omega_1})\vec X_i\},S_r, \vec 1,N)\nonumber\\
&=&(\{f_i(x_1,\cdots,x_m)\},\boldsymbol{\Omega_1},\{\boldsymbol{Q_i}(\boldsymbol{\Omega_1})\vec X_i\}, \vec 1,N),
\end{eqnarray}
or, alternatively, the uniform mean field PEDS are ordering independent.
\end{corollary}
\begin{proof}
The proof follows directly from the fact that any analytic function $f_i(x_1,\cdots,x_m)$ can be written in the form of a series expansion as in Proposition \ref{prop:independence}.
\end{proof}

This implies essentially that for any dynamical system, we can write the dynamics with the most convenient ordering, without affecting the dynamics.

\section{Numerical Examples} \label{sec:numerics}
We present here some numerical examples of application of the PEDS procedure. 

\subsection{Implementation remarks}
For the sake of implementation, it would be convenient to define the PEDS transformation without having to evaluate the Taylor expansion, as for the theoretical developments in the previous Sections. This can be easily carried out for factorized vector target systems $f_i(x_1,\cdots,x_m)=\prod_{k=1}^m f_{i,k}(x_k)$ (or linear combinations of factorized terms of the same type). For such factorized target systems, the matrix map can be built as the function $\boldsymbol{F}_{i,k}(\boldsymbol{\Omega}\boldsymbol{X}_k)$ (and \ $\boldsymbol{F}_i(\boldsymbol{X}_1,\cdots,\boldsymbol{X}_m)=\sum_i a_i\prod_{k=1}^m \boldsymbol{F}_{i,k}(\boldsymbol{X}_k)$ or similar expressions).

The question is therefore how to efficiently evaluate such matrix functions. This can be done defining the matrix maps as the Taylor expansion evaluated in matrix $\boldsymbol{\Omega} \boldsymbol{X}_k $. We can write
\begin{equation}
    (\boldsymbol{\Omega} \boldsymbol{X}_k)^s=\sqrt{\boldsymbol{X}_k}^{-1} \big(\sqrt{\boldsymbol{X}_k}\boldsymbol{\Omega} \sqrt{\boldsymbol{X}_k} \big)^s \sqrt{\boldsymbol{X}_k}\label{eq:symom}
\end{equation}
where $\sqrt{\boldsymbol{X}_k}$ always exists since $\boldsymbol{X}_k$ is diagonal. Notice that \eqref{eq:symom} defines a similarity transformation, i.e. it conserves the spectrum of the similar matrices. An important point is to verify that if the spectrum of $\boldsymbol{\Omega} \boldsymbol{X}_k$ is real and $\boldsymbol{X}_k$ is diagonal, then the spectrum of $\sqrt{\boldsymbol{X}_k}\boldsymbol{\Omega} \sqrt{\boldsymbol{X}_k}$ is also real. This can be shown using the fact that, based on the definition of the Cayley polynomial and on the determinant properties, the eigenvalue problem for $\boldsymbol{\Omega} \boldsymbol{X}_k$ is equivalent to the generalized eigenvalue problem $\boldsymbol{\Omega}\vec v-\lambda \boldsymbol{X}_k^{-1}\vec v=\vec 0$, assuming $\boldsymbol{X}_k$ invertible. Then, a proof similar to the spectral theorem shows that if $\boldsymbol{\Omega}$ and $\boldsymbol{X}_k^{-1}$ are symmetric and real, then $\text{Im}(\lambda)=0$. This guarantees that an extension to the complex field of the scalar target system functions $f_{i,k}(x_k)$ is not required. In fact, \eqref{eq:symom} implies
\begin{eqnarray}
\boldsymbol{F}_{i,k}(\boldsymbol{\Omega} \boldsymbol{X}_k)=\sqrt{\boldsymbol{X}_k}^{-1}\boldsymbol{F}_{i,k}(\sqrt{\boldsymbol{X}_k}\boldsymbol{\Omega} \sqrt{\boldsymbol{X}_k}) \sqrt{\boldsymbol{X}_k}.
\end{eqnarray}
Since $\sqrt{\boldsymbol{X}_k}\boldsymbol{\Omega} \sqrt{\boldsymbol{X}_k}$ is symmetric,  $\boldsymbol{P}_{x_k}$ exists such that  $$\sqrt{\boldsymbol{X}_k}\boldsymbol{\Omega} \sqrt{\boldsymbol{X}_k}=\boldsymbol{P}_{x_k} \boldsymbol{\Sigma}_{x_k}\boldsymbol{P}_{x_k}^{-1}$$ where the real matrix $\boldsymbol{\Sigma}_{x_k}=\text{diag}\{\sigma_{x_k,1},\cdots,\sigma_{x_k,N}\}$ is made of the elements of the spectrum of $\sqrt{\boldsymbol{X}_k}\boldsymbol{\Omega} \sqrt{\boldsymbol{X}_k}$. As a result, we can write
\begin{eqnarray}
    \boldsymbol{F}_{i,k}(\boldsymbol{\Omega} \boldsymbol{X}_k)=\sqrt{\boldsymbol{X}_k}^{-1}\boldsymbol{P}_{x_1}\boldsymbol{F}_{i,k}(\boldsymbol{\Sigma}_{x_k})\boldsymbol{P}_{x_k}^{-1}\sqrt{\boldsymbol{X}_k}
\end{eqnarray}
where 
\begin{eqnarray}
    \boldsymbol{F}_{i,k}(\boldsymbol{\Sigma}_{x_1})=\text{diag}\big(f_{i,k}(\sigma_{x_k,1}),\cdots,f_{i,k}(\sigma_{x_k,N}) \big).
\end{eqnarray}
Thus, evaluating the matrix maps boils down to the knowledge of the eigenvalues and eigenvectors of $\sqrt{\boldsymbol{X}_k}\boldsymbol{\Omega} \sqrt{\boldsymbol{X}_k}$. The question is whether these two must evaluated at every time step, as $\sqrt{\boldsymbol{X}_k}$ is a dynamical variable: unfortunately this is the case.

\subsection{Uniform mean field projector}
We now provide several examples to show the applications of the theory developed in this paper.
\subsubsection{One dimensional potential}

Our first example is a one dimensional nonlinear dynamical system, written as:
\begin{equation}
    \frac{dx}{dt}=f(x)=-\frac{\partial V(x)}{\partial x}.
    \label{eq:dyn1d}
\end{equation}
Let us thus analyze numerically the PEDS $\mathcal O=\{f(x),\boldsymbol{\Omega}_1,-\alpha (\boldsymbol{I}-\boldsymbol{\Omega}_1) \vec X,\vec 1,N\}$, with a potential of the form 
\begin{eqnarray}
V(x)=a_0+a_1 x+\frac{a_2}{2} x^2 +\frac{a_3}{3} x^3 + \frac{a_4}{4} x^4
\label{eq:potential}
\end{eqnarray}
for a set of parameters for which two minima are present, as shown in Fig.~\ref{fig:pot}. First, we compare the standard commutative and non-commutative maps. The difference is shown in Fig.~\ref{fig:linnoncomm} for identical initial conditions.

\begin{figure}
    \centering
    \includegraphics[scale=.15]{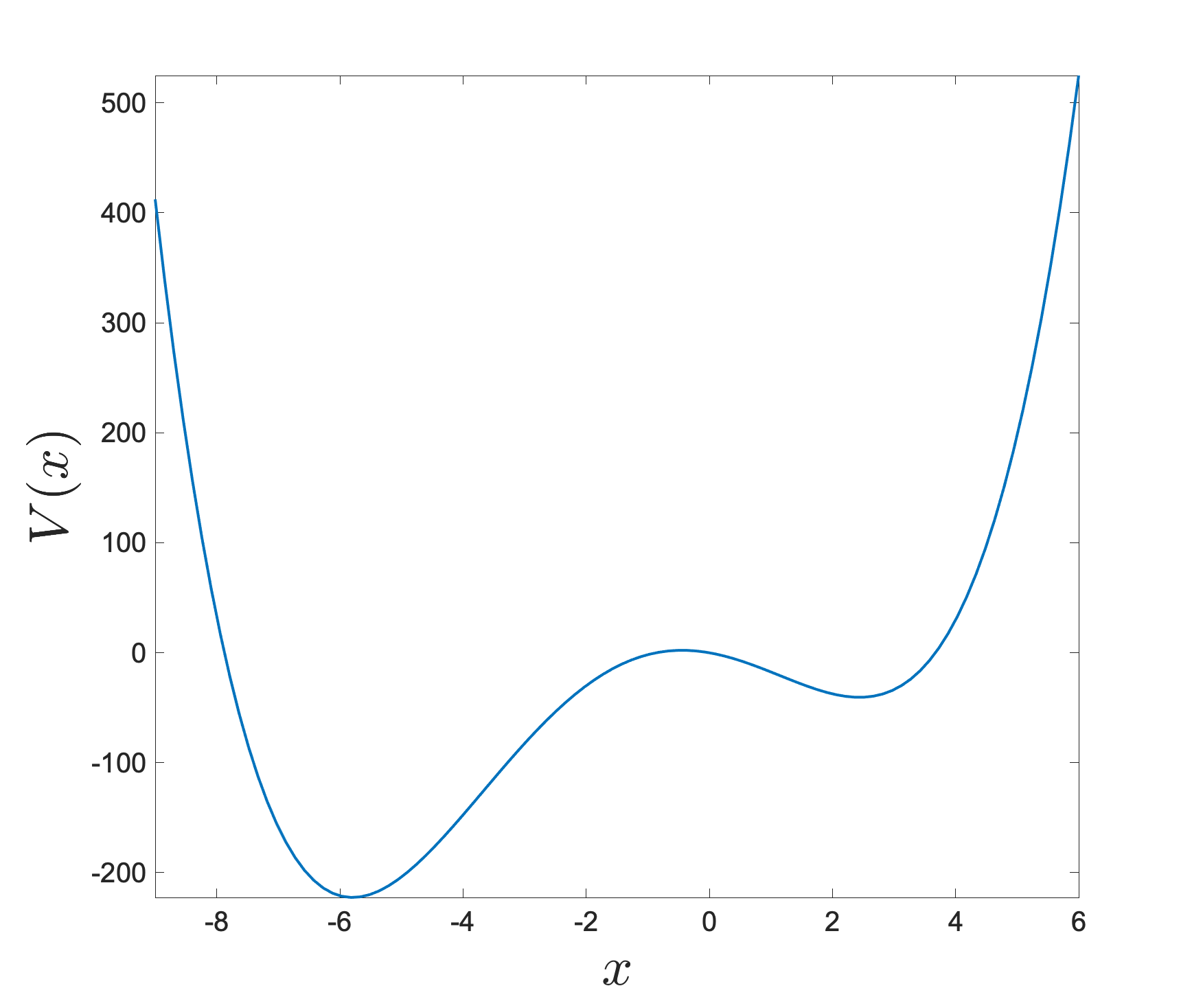}
    \caption{Potential $V(x)$ in \eqref{eq:potential}, for $a_0=0$, $a_1=9.85$, $a_2=10$, $a_3=2$ and $a_4=-0.395$. The parameters are chosen to provide the potential minima.}
    \label{fig:pot}
\end{figure}

The associated PEDS is given by the differential system
\begin{eqnarray}
    \frac{d\vec X}{dt}&=&\boldsymbol{F}(\vec X)\vec 1-\alpha (\boldsymbol{I}-\boldsymbol{\Omega_1}) \vec X\nonumber \\
    &=&-\Big(a_1\boldsymbol{I}+a_2 (\boldsymbol{\Omega_1} \boldsymbol{X})+a_3 (\boldsymbol{\Omega_1} \boldsymbol{X})^2+a_4  (\boldsymbol{\Omega_1} \boldsymbol{X})^3 \Big) \vec 1-\alpha (\boldsymbol{I}- \boldsymbol{\Omega_1})\vec 1
\end{eqnarray}

\begin{figure}
    \centering
   \includegraphics[scale=.25]{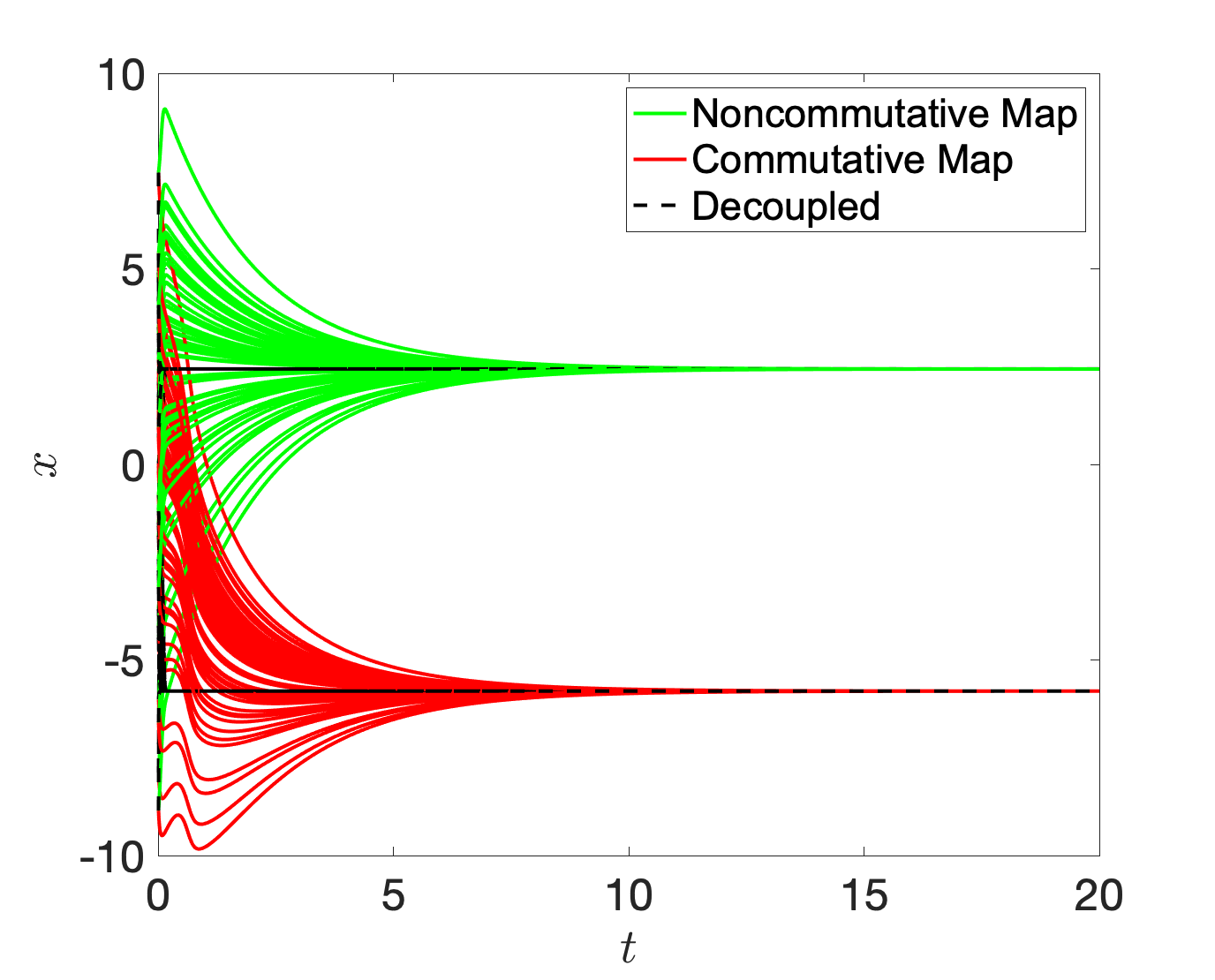}
    \caption{Comparison between the standard commutative and non-commutative maps for the 1D potential \eqref{eq:potential} and the target system.  }
    \label{fig:linnoncomm}
\end{figure}

\begin{figure}
    \centering
    \includegraphics[scale=.2]{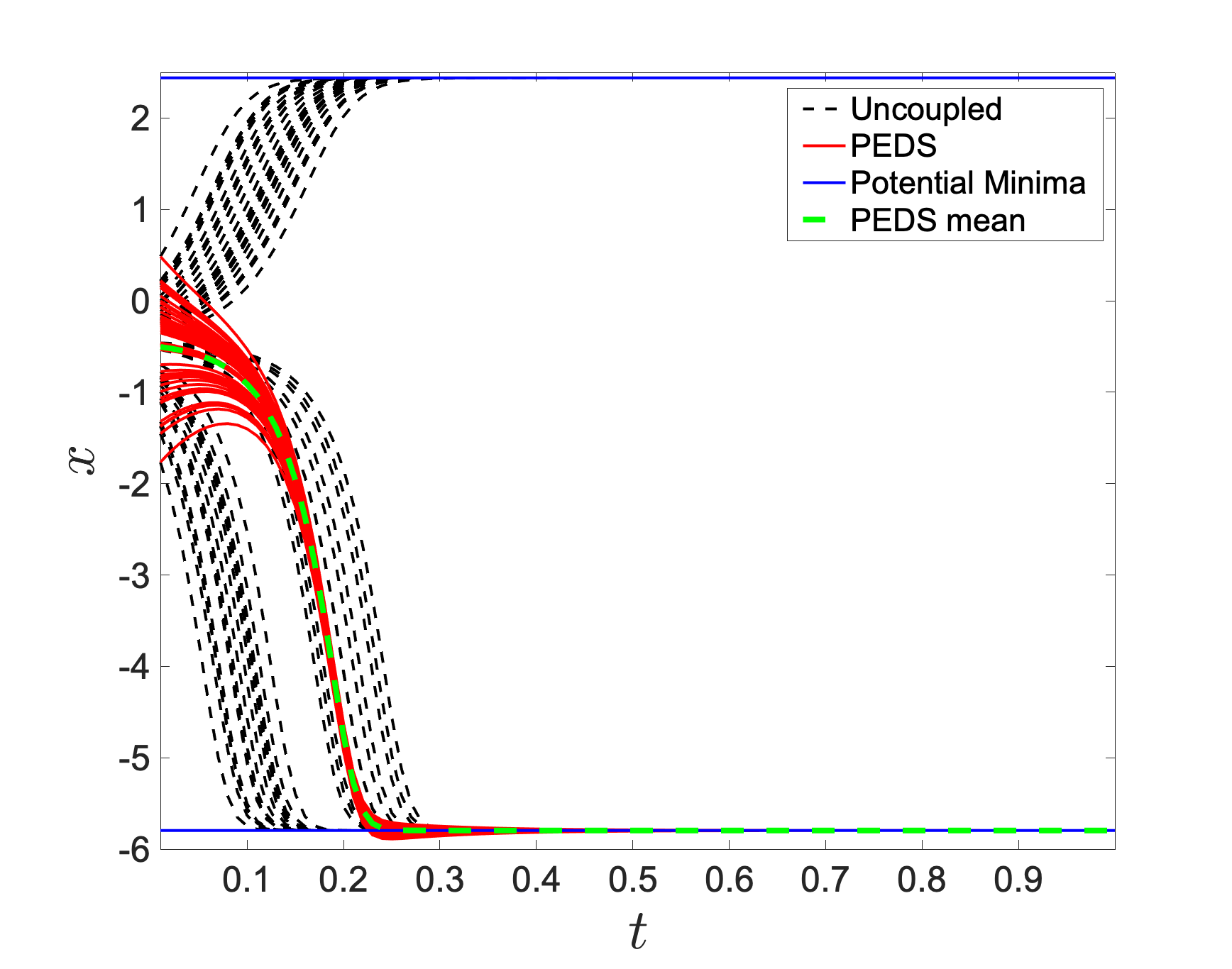}
    \caption{Evolution of the PEDS $\{f(x),\boldsymbol{\Omega}_1,-\alpha (\boldsymbol{I}-\boldsymbol{\Omega_1}) \vec X,\vec 1,N\}$ for random initial conditions around the maximum for random initial conditions for the non-commutative map. The blue solid lines are the target system minima. The black dashed curves are the target system trajectories ($\Omega=I$), that asymptotically reach both the two minima. The red curves refer to the coupled PEDS with the uniform mean field projector using the standard non-commutative map, leading to the global minimum, as it can be seen from the plot of the $\tilde x(t)=\mathcal P_{\boldsymbol{\Omega}_1} \vec X(t)$, the green dashed line. 
    \label{fig:thesol}}
\end{figure}

The results of the numerical integration, using a simple Euler scheme for Gaussian-distributed initial conditions around the potential maximum, at $x^*=-0.51$, of the target system and of the PEDS embedding with the standard non-commutative map. The PEDS trajectories all reach the global minimum of the potential $V(x)$, while the uncoupled trajectories split between the two stable equilibria.

\subsubsection{Vector target system}
As an example of vector target system, we consider a two-dimensional dynamical system embedded with the standard non-commutative map.
The target system is:
\begin{eqnarray}
    \frac{dx}{dt}&=&-\frac{\partial V(x,y)}{\partial x}\\
    \frac{dy}{dt}&=&-\frac{\partial V(x,y)}{\partial y}
\end{eqnarray}
where
\begin{equation}
    V(x,y)=\exp\left(\frac{x^2}{2}-\frac{y^2}{2}+\frac{y^4}{4}\right)
    \label{eq:vxy}
\end{equation}
which is characterized by two local minima,$(x^*=0,y^*=\pm 1)$.
The equations of motion define the gradient descent dynamics
\begin{eqnarray}
    \frac{dx}{dt}&=&f_x(x,y)=-x V(x,y)\\
    \frac{dy}{dt}&=&f_y(x,y)=-(y-y^3) V(x,y).
\end{eqnarray}

The interest in this examples lies in the fact that the PEDS equations of motion depend on the ordering prescription considered. We discuss here the two cases defined below:
\begin{eqnarray}
    S_1&\rightarrow& \begin{cases}\boldsymbol{F}^{(1)}_x(\vec X,\vec Y)&=-\boldsymbol{\Omega}_1 \boldsymbol{X} \ \boldsymbol{V}_1(\boldsymbol{X},\boldsymbol{Y}) \\
    \boldsymbol{F}_y^{(1)}(\vec X,\vec Y)&= -(\boldsymbol{\Omega}_1 \boldsymbol{Y})\big(\boldsymbol{I}-(\boldsymbol{\Omega}_1 \boldsymbol{Y})^2 \big)\ \boldsymbol{V}_1(\boldsymbol{X},\boldsymbol{Y}) \\
    \end{cases}  \\
       S_2&\rightarrow & \begin{cases}\boldsymbol{F}_x^{(2)}(\vec X,\vec Y)&=-\boldsymbol{\Omega}_1 \boldsymbol{X}\ \boldsymbol{V}_2(\boldsymbol{X},\boldsymbol{Y}) \\
    \boldsymbol{F}_y^{(2)}(\vec X,\vec Y)&= -(\boldsymbol{\Omega}_1 \boldsymbol{Y})\big(\boldsymbol{I}-(\boldsymbol{\Omega}_1 \boldsymbol{Y})^2 \big)\ \boldsymbol{V}_2(\boldsymbol{X},\boldsymbol{Y})  \\
    \end{cases}
\end{eqnarray}
where
\begin{eqnarray}
    \boldsymbol{V}_1(\boldsymbol{X},\boldsymbol{Y})&=&\text{exp}\left(\frac{1}{2}(\boldsymbol{\Omega_1} \boldsymbol{X})^2+(\boldsymbol{\Omega_1} \boldsymbol{Y})^2(\frac{1}{2}\boldsymbol{I}-\frac{1}{4}(\boldsymbol{\Omega_1} \boldsymbol{Y})^2 \right)\\
    \boldsymbol{V}_2(\boldsymbol{X},\boldsymbol{Y})&=&\text{exp}\Big(\frac{1}{2}(\boldsymbol{\Omega_1} \boldsymbol{X})^2\Big)\text{exp}\Big((\boldsymbol{\Omega_1} \boldsymbol{Y})^2(\frac{1}{2}\boldsymbol{I}-\frac{1}{4}(\boldsymbol{\Omega_1} \boldsymbol{Y})^2 \big)\Big)
\end{eqnarray}
where $\boldsymbol{V}_1\neq \boldsymbol{V}_2$ since  $\exp(\boldsymbol{\Omega}_1 \boldsymbol{X}\boldsymbol{\Omega}_1 \boldsymbol{Y})\neq \exp(\boldsymbol{\Omega}_1 \boldsymbol{X})\text{exp}(\boldsymbol{\Omega}_1 \boldsymbol{Y})$. Thus, choosing one versus the other is equivalent to a different ordering  choice.

\begin{figure}
    \centering
    \includegraphics[scale=0.2]{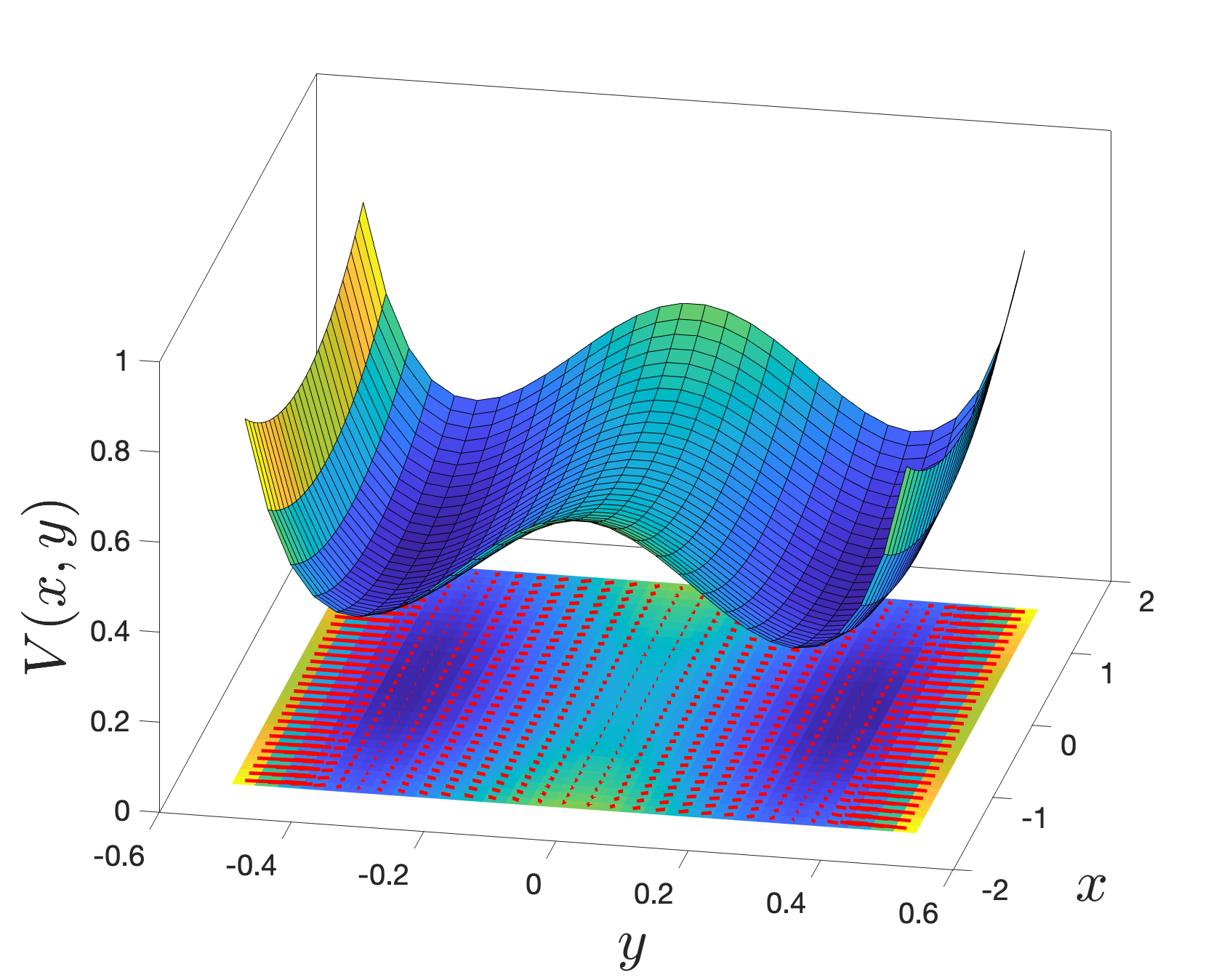}
    \caption{The potential $V(x,y)$ in \eqref{eq:vxy}, the projection with the two minima.}
    \label{fig:pot2d}
\end{figure}

We embed this system of equations via $\mathcal O_1=(\{f_x,f_y\},\boldsymbol{\Omega_1},\{-\alpha (\boldsymbol{I}-\boldsymbol{\Omega}_1)\vec X,-\alpha(\boldsymbol{I}-\boldsymbol{\Omega}_1)\vec Y\},S_1,\vec 1,N)$ and
$\mathcal O_2=(\{f_x,f_y\},\boldsymbol{\Omega_1},\{-\alpha (\boldsymbol{I}-\boldsymbol{\Omega}_1)\vec X,-\alpha(\boldsymbol{I}-\boldsymbol{\Omega}_1)\vec Y\},S_2,\vec 1,N)$, obtaining
\begin{eqnarray}
    \frac{d\vec X}{dt}&=&\boldsymbol{\Omega}_1\boldsymbol{F}_x^{(i)}\vec 1-\alpha (\boldsymbol{I}-\boldsymbol{\Omega}_1)\vec X,\\
    \frac{d\vec Y}{dt}&=&\boldsymbol{\Omega}_1\boldsymbol{F}_y^{(i)}\vec 1-\alpha (\boldsymbol{I}-\boldsymbol{\Omega}_1)\vec Y,
\end{eqnarray}
where $i=1,2$ is the label for $S_1,S_2$.
The numerical solutions are shown in Fig.~\ref{fig:sol2d} using $\alpha=0.1$ and $N=50$. The two dynamic behaviours are essentially identical, as per  Corollary~\ref{cor:independence}. Interestingly, even if in general $[\boldsymbol{\Omega_1 X},\boldsymbol{\Omega_1 Y}]\neq 0$, for this example we can work out the full details leading to the independence on the ordering.

First, let us note that $\boldsymbol{\Omega_1 X \Omega_1}=\langle X\rangle \boldsymbol{\Omega}_1 $. Using this expression, we can write $$f(\boldsymbol{\Omega}_1 \boldsymbol{X})= \boldsymbol{I}+\frac{f(\langle X \rangle) -1}{\langle X \rangle}\boldsymbol{\Omega}_1 \boldsymbol{X}$$ therefore
\begin{align}
\exp\left( f(\boldsymbol{\Omega}_1 \boldsymbol{X})+g(\boldsymbol{\Omega}_1 \boldsymbol{Y}) \right) &= \exp\left( 2\boldsymbol{I} + \frac{f(\langle X\rangle)-1}{\langle X\rangle}\boldsymbol{\Omega}_1 \boldsymbol{X}+\frac{g(\langle Y\rangle)-1}{\langle Y\rangle}\boldsymbol{\Omega}_1 \boldsymbol{Y} \right) \nonumber\\
&=\exp(2)\exp\left[ \boldsymbol{\Omega}_1 \left(\frac{f(\langle X\rangle)-1}{\langle X\rangle}\boldsymbol{ X}+\frac{g(\langle Y\rangle)-1}{\langle Y\rangle}\boldsymbol{ Y} \right) \right]
\end{align}
We can now apply again the same formula
\begin{align}
&\exp\left[ \boldsymbol{\Omega}_1 \left(\frac{f(\langle X\rangle)-1}{\langle X\rangle}\boldsymbol{ X}+\frac{g(\langle Y\rangle)-1}{\langle Y\rangle}\boldsymbol{ Y} \right) \right] \nonumber\\
&\quad= \exp\left[ \boldsymbol{I}+\frac{{f(\langle X\rangle)+g(\langle Y\rangle)-2}-1}{f(\langle X\rangle)+g(\langle Y\rangle)-2}\boldsymbol{\Omega}_1 \left(\frac{f(\langle X\rangle)-1}{\langle X\rangle}\boldsymbol{ X}+\frac{g(\langle Y\rangle)-1}{\langle Y\rangle}\boldsymbol{ Y}\right) \right]
\end{align}
yielding, after taking into account the multiplication times  $\vec 1$
\begin{align}
&\exp(2)\exp\left[ \boldsymbol{I}+\frac{{f(\langle X\rangle)+g(\langle Y\rangle)-2}-1}{f(\langle X\rangle)+g(\langle Y\rangle)-2}\boldsymbol{\Omega}_1 \left(\frac{f(\langle X\rangle)-1}{\langle X\rangle}\boldsymbol{ X}+\frac{g(\langle Y\rangle)-1}{\langle Y\rangle}\boldsymbol{ Y}\right) \right]\vec 1\nonumber\\
&\quad=\exp\left( f(\langle X\rangle)+g(\langle Y\rangle)  \right) \vec 1.
\end{align}
In other words, we have shown that
\begin{equation}
    \exp\left( f(\boldsymbol{\Omega}_1 \boldsymbol{X}) + g(\boldsymbol{\Omega}_1 \boldsymbol{Y}) \right) \vec 1 = \exp\left( f(\boldsymbol{\Omega}_1 \boldsymbol{X}) \right)\exp\left( g(\boldsymbol{\Omega}_1 \boldsymbol{Y}) \right) \vec 1,
\end{equation}
i.e. the equivalence of the dynamics of $S_1$ and $S_2$.
\begin{figure}
    \centering
    \includegraphics[scale=0.22]{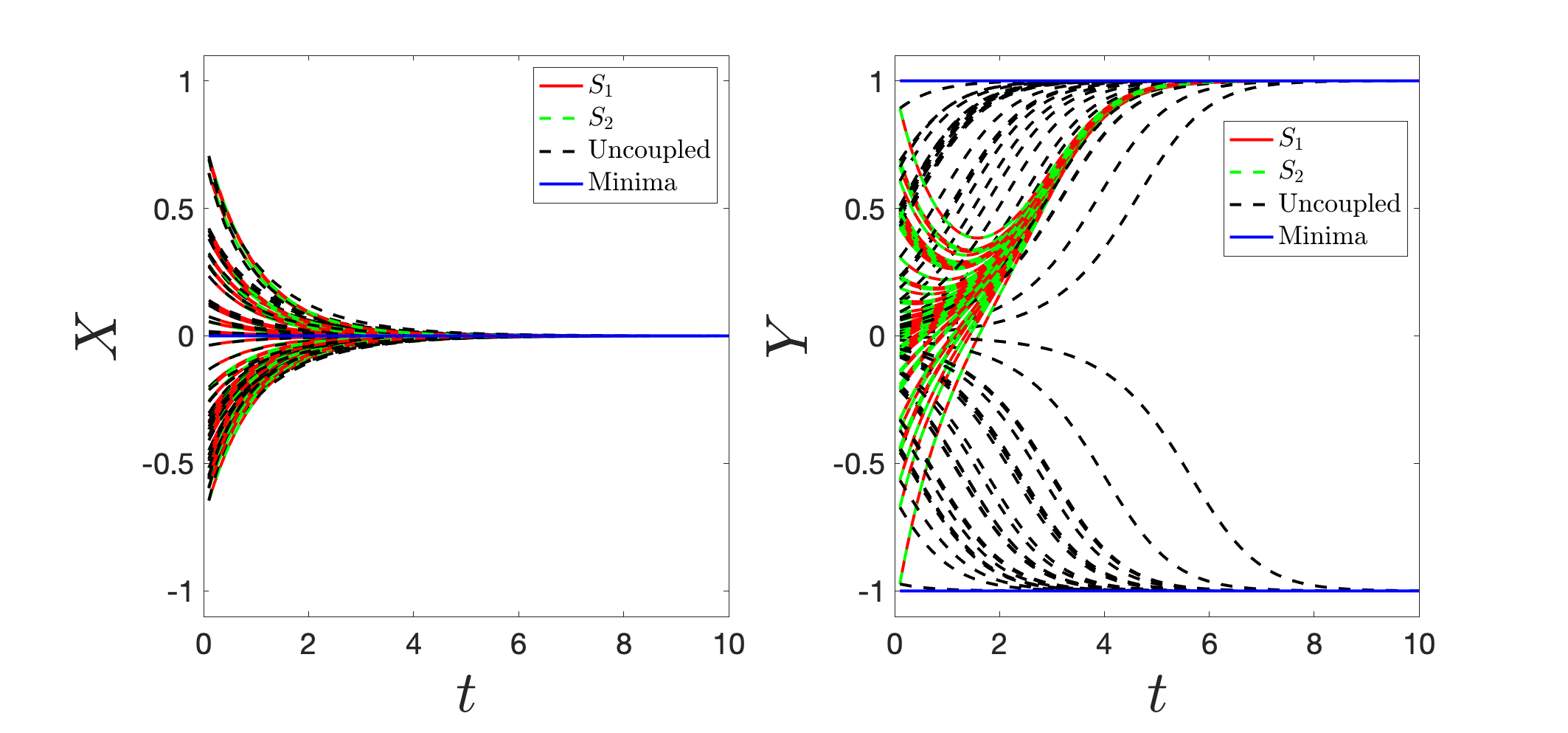}
    \caption{The two PEDS dynamics $\mathcal O_1$ and $\mathcal O_2$ for the $x$ and $y$ variables (left and right) for $N=50$ and integrated using an Euler scheme with step size $dt=0.1$. We consider here the noncommutative map. We can see that the dynamics of $S_1$ and $S_2$ are identical. Note that both $S_1$, $S_2$ and the uncoupled systems have the identical initial conditions.}
    \label{fig:sol2d}
\end{figure}

\subsubsection{Hamiltonian equations with dissipation}

As a third example, let us consider another two-dimensional vector target system: the description of a dissipative Hamiltonian system for a single particle of mass $m$. The target system reads
\begin{eqnarray}
    \frac{dx}{dt}&=&\frac{p}{m}\\
    \frac{dp}{dt}&=&-\frac{\partial V}{\partial x} -\chi \frac{p}{m}
\end{eqnarray}
where $\chi$ denotes the dissipation and we define the force $f(x,p)=-\partial V/\partial x$. Following the prescription of the previous sections, we write the PEDS
$$\mathcal O =(\{\frac{p}{m},f(x,p)-\chi {p}/{m}\},\boldsymbol{\Omega_1},\{-\alpha_x (\boldsymbol{I-\Omega_1})\vec X,-\alpha_p (\boldsymbol{I-\Omega_1})\vec P\},\vec 1,N).$$ The extended system equations are given by
\begin{eqnarray}
    \frac{d \vec X}{dt}&=& \frac{1}{m} \boldsymbol{\Omega}_1 \boldsymbol{P} \vec 1 - \alpha_x (\boldsymbol{I}-\boldsymbol{\Omega}_1) \vec X\\
    \frac{d \vec P}{dt}&=&\boldsymbol{\Omega}_1 \left(\boldsymbol{F}( \boldsymbol{ X})-\chi  \frac{\boldsymbol{P}}{m} \right)\vec 1- \alpha_p (\boldsymbol{I}-\boldsymbol{\Omega}_1) \vec P
\end{eqnarray}
which is thus a $2N$ set of equations. Let us focus on the potential $V(x,p)$ defined in \eqref{eq:potential}. 
The results are shown in Fig.~\ref{fig:solham2d}.

\begin{figure}
    \centering
    \includegraphics[scale=0.2]{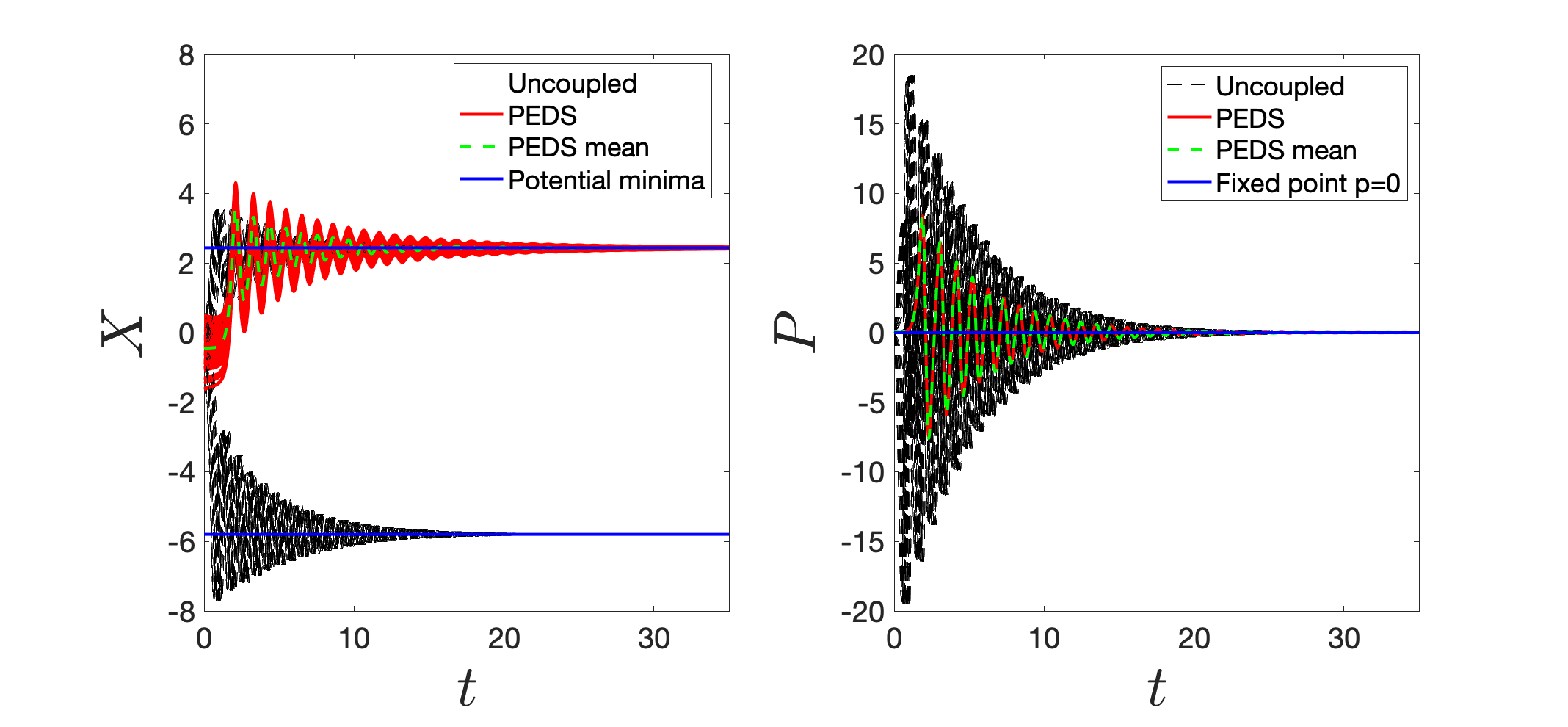}
    \caption{PEDS dynamics for the Hamiltonian system with one particle, integrated using an Euler scheme with step size $dt=0.001$.  The extended system has $N=50$. Also in this case, the dynamics approaches the fixed points of the potential for the $x$ variable, reaching a local minimum, while the $p$ variables tends to the asymptotic value $p^*=0$.}
    \label{fig:solham2d}
\end{figure}

\subsection{Beyond the uniform mean field projector}
This paper is focused mostly on the uniform mean field projection $\boldsymbol{\Omega}_1$. Before concluding, we wish to numerically simulate also the case of a PEDS with a different projection operator. Let us consider a PEDS where  $\boldsymbol{\Omega}=\boldsymbol{B}^t(\boldsymbol{B}\boldsymbol{B}^t)^{-1}\boldsymbol{B}$, where $\boldsymbol{B}$ is a random square matrix (uniformly distributed on $[0,1]$) of size $N\times K$. The scalar target dynamical system we are interested in is again
\begin{equation}
    \frac{dx}{dt}=f(x)=-\frac{\partial V(x)}{\partial x}.
    \label{eq:dyn1dr}
\end{equation}
with the potential in \eqref{eq:potential}, choosing parameters $a_4=a_0=0$, $a_3=-2$, $a_2=-10$ and $a_1=9.85$ that guarantee a single potential minimum, as shown Fig. \ref{fig:randomproj} (left). We then consider the PEDS
 $\mathcal O =(f(x),\boldsymbol{\Omega},-\alpha (\boldsymbol{I-\Omega})\vec X,\boldsymbol{\Omega}\vec 1,N)$, with $N=50$, and follow the observable $\tilde x=\mathcal P_{\boldsymbol{\Omega}} \vec X$.
 The results are shown in Fig. \ref{fig:randomproj} (right). The PEDS embedding converges also in this case to the potential absolute minimum, thus confirming that a generalization of the results of this paper to arbitrary projectors is possible.
 
\begin{figure}
    \centering
    \includegraphics[scale=0.2]{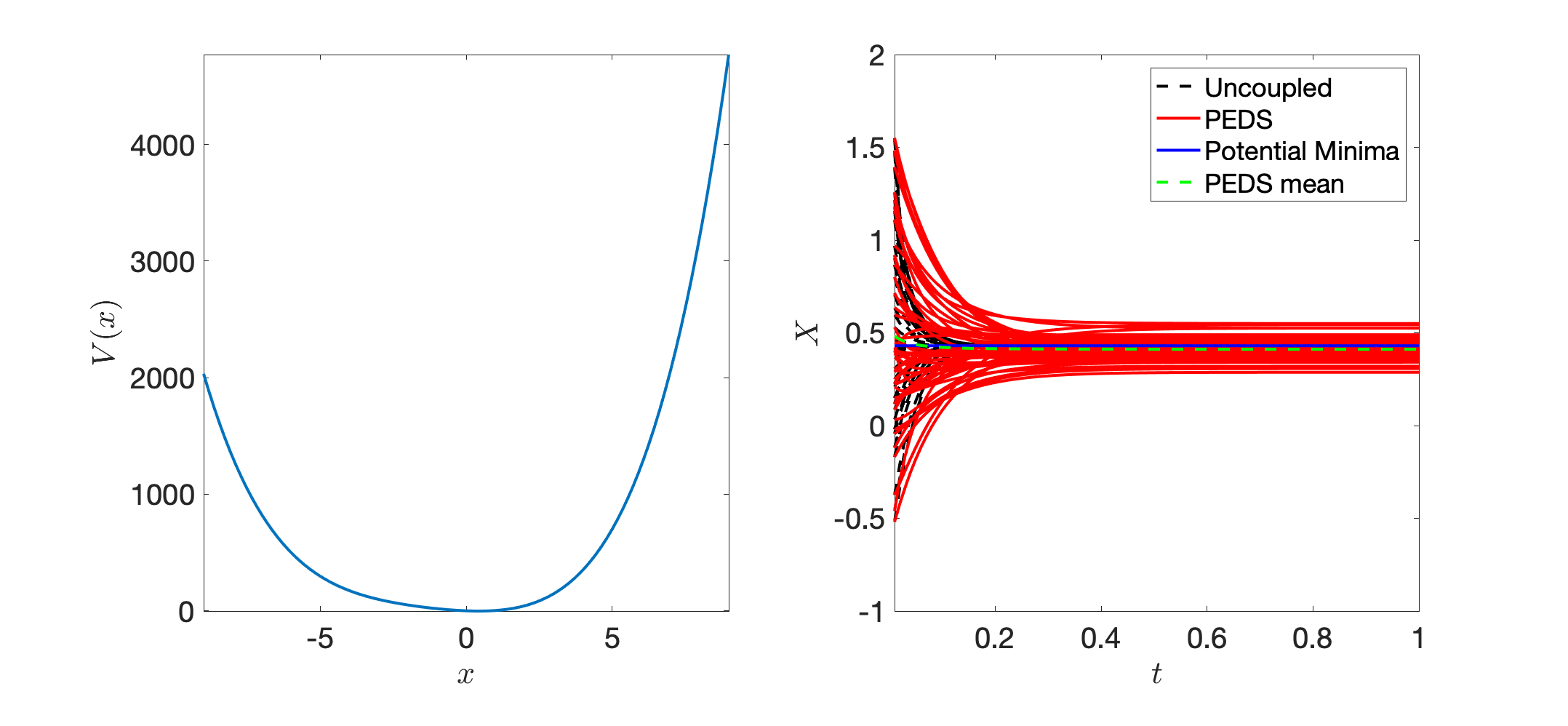}
    \caption{PEDS dynamics for the target system \eqref{eq:dyn1dr} with a random projection operator. We can see that the results we established for this paper in the case $\boldsymbol{\Omega_1}$ could be possibly extended to more general projectors.}
    \label{fig:randomproj}
\end{figure}

As a last comment, one of the main motivations for this study is that in circuits, conservation laws can be expressed in terms of projector operators. An example is the volatile but (almost) ideal memristor.
A resistor with memory can be described, at the lowest level of approximation for a current controlled device, by an effective dynamical resistance depending on an internal parameter $x$. In this sense, memristors are approximately described by the functional form $R(x)=R_\text{off} (1-x)+x R_\text{on}$, where $R_\text{on}<R_\text{off}$ are the boundary resistances, and $x\in[0,1]$. We assume that the internal memory parameter $x$ evolves according to a simple equation of the form $dx/{dt}={R_\text{on}}I/{\beta}-\alpha x$. 
 The parameters $\alpha$ and $\beta$ are the decay constant and the effective activation voltage per unit time, respectively. For a recent paper which inspired this study, consider \cite{caravelliscience}, where  transitions between effective minima of a lower dimensional potential were observed. Using Ohm's law, we define voltage $V=R(x) I$, so as to obtain a normalized equation for $x(t)$
\begin{eqnarray}
 \frac{dx}{dt}=\frac{V}{ \beta} \frac{1}{1-\chi x}-\alpha x=-\alpha \frac{\partial V(x,s)}{\partial x}
 \label{eq:oned}
\end{eqnarray}
where $\chi={(R_\text{off}-R_\text{on})}/{R_\text{off}}$ and $s=\frac{S}{\alpha \beta}$, with $0\leq \chi \leq 1$ in the physically relevant cases, and $V(x, s)$ as an effective potential, where $S$ is the voltage applied to the circuit, and $s$ is a normalized quantity with units of inverse time.

The dynamics of a single memristor \eqref{eq:oned} is fully characterized by the gradient following the dynamics of the effective potential
\begin{equation}
V(x,s)=\frac{1}{2} x^2 +\frac{s}{\chi} \log(1-\chi x), 
\label{eq:potential2}
\end{equation}
with $s=\frac{S}{\alpha \beta}$; the constant $\alpha$ also acts as the learning rate in \eqref{eq:oned}. 

For a network of memristors, the differential equation for $x_i(t)$ is a set of coupled ODE of the form \cite{Caravelli2016rl}:
\begin{eqnarray}
\frac{d}{dt} \vec x=\frac{1}{\beta}(\boldsymbol{I}-\chi \boldsymbol{\Omega X})^{-1} \boldsymbol{\Omega} \vec S-\alpha \vec x,
\label{eq:manyd}
\end{eqnarray}
where $\boldsymbol{X}_{ij}(t)=x_i(t) \delta_{ij}$. The matrix $\boldsymbol{\Omega}$ is the projection operator on the vector space of cycles of $\mathcal G$,
the graph representing the circuit~\cite{Caravelli2016rl}, and, as discussed in the Introduction, a mathematical consequence of Kirchhoff's conservation laws. 
Now, we note that we can write \eqref{eq:manyd} as:
\begin{eqnarray}
\frac{d}{dt} \vec x=\boldsymbol{\Omega}(\frac{1}{\beta}(\boldsymbol{I}-\chi \boldsymbol{\Omega X})^{-1} \boldsymbol{\Omega} \vec S-\alpha \vec x)-\alpha (\boldsymbol{I}- \boldsymbol{\Omega})\vec x,
\label{eq:manyd2}
\end{eqnarray}
which is exactly in the form of a PEDS, with a standard decay function. Thus, the results of \cite{caravelliscience} can be interpreted as the relaxation of the system towards the minima defined by the embedding function. If $\boldsymbol{\Omega}=\boldsymbol{\Omega}_1$, then using the results of this paper we know that the potential \eqref{eq:potential2} determines the effective minima of the system. However, in order to justify the presence of the  rumbling transitions shown in \cite{caravelliscience}, a deeper understanding of the PEDS properties for a general projector $\boldsymbol{\Omega}$ is required.

\section{Conclusions and perspective}
In the present paper we presented and  studied a map between dynamical systems of size $m$ and  dynamical systems in a higher number of variables. This is the first of a series of papers formally investigating the projective embeddings of dynamical systems (PEDS) paradigm that we defined here. The purpose of this work was to formally show the properties of this type of embeddings, within the context of a particular projector matrix. As we have seen, their structure is such that for long times, the asymptotic equilibria of the target dynamical system can be recovered.  

We have discussed in particular the case of the uniform mean field projector operator ${\Omega}_{1,ij}=\frac{1}{N}$. For this choice, we have been able to prove analytically that the asymptotic equilibria are strictly connected to those of the original system. Aside from establishing the formalism, this paper also established some exact results about how the embedding changes the properties of the dynamics critical, including the cases of unstable equilibria and saddle points.

Specifically, we have studied the embedding  of $m$ dimensional dynamical systems in $Nm$-dimensional systems. The purpose of such embedding is to modify the nature of the fixed points of the dynamics, i.e. those satisfying $\{ \vec x^*\ \text{s.t.}\ {d\vec x}/{dt}|_{\vec x^*}=0\}$. In particular, we have shown that stable and saddle type fixed points retain their properties, while unstable fixed points become saddles. This observation justify future works in this direction, in particular exploiting different types of decay functions, matrix embeddings and projectors with respect to this contribution. It is worth to mention that a follow up of this work is in preparation, in which we discuss the behavior of PEDS for general projectors; many of the results on the Jacobian obtain in this paper do actually apply also in the general case \cite{PEDS2}.

An important aspect of interest of future works will be to focus on how to further modify the spectral properties of the fixed points, i.e. the nature of the Jacobian once evaluated at $\vec x^*$. What we have shown in the present paper is that, for the uniform projector, the PEDS Jacobian is always symmetric, and thus characterized by real eigenvalues, that in particular are negative if the corresponding fixed point of the target system is stable. This implies that the dynamics near stable fixed points is always laminar, e.g. slowly decaying towards the fixed point. As we will see in future works, this is not the case for general projectors, for which approximate but special techniques will have to be employed. 

As discussed, the spectral signature is in part inherited by the original, target dynamical system, but modified through the extended number of dimensions. The idea of generalizing the space of solutions to higher dimensions is not new. In a way, the PEDS technique is in spirit close to both Markov Chain Monte Carlo methods \cite{mcmc} and the notion of lifts in convex optimization \cite{lifts}, but is specifically developed for the fixed points of dynamical systems.

In particular, in \cite{caravelliscience} it was observed that memristive circuits have an effective lower dimensional representation in terms of an effective potential, and that they can exhibit a ``rumbling'' transition, i.e. a transient chaotic tunneling between local minima of a properly defined potential. As it turns out, such dynamics is only a particular case of the PEDS introduced here, in which the projector operator was given by random circuit connections.


The rumbling transition in \cite{caravelliscience} was pinpointed numerically to be due to an effective ``Lyapunov force'', shown to be present in connection with the rumbling transition phenomenon.  Such force was defined essentially as a deviation from a mean field theory, and we provided evidence of an athermal and novel mechanism in which barrier escapes emerge in the effective description of a multi-particle system.
This paper is a continuation of that work, attempting at generalizing those findings to general systems, although focusing specifically on a particular type of projector: in this case, these ``Lyapunov" forces are not present. Similar yet different types of behavior were also observed previously  within the context of memory-based computing (memcomputing) solutions \cite{reviewCarCar,Sheldon,traversapol,sciadvtra,Inst1,Inst2}.

The main focus of this paper represents a first step towards a clarification of the general reasons why the introduction of hidden variables in a dynamical system
can lead to transitions between local and global minima of the effective description via instabilities in the full system. Since maxima can be turned into saddle points, generically there cannot be no ``barriers'' when the target system is a gradient descent. However, as we will show formally in future works, in order to obtain the rumbling transitions, one has to go beyond the uniform mean field approximation and study a more general type of projector.


Clearly, the projective embedding studied in this paper can be employed in a variety of dynamical systems, including all sort of gradient-based dynamics, with applicability to machine learning and neural networks. These applications will also be the subject of future studies. In particular, we hope that the introduction of ``hidden variables'' in dynamical systems \cite{Bohm} can be further investigated for the purpose of machine learning and optimization applications \cite{ganguli}. In general, the study of transient chaos in dynamical systems and optimization \cite{toroc0,tchaos} is an interesting area of research with possible applications also in memristor-based algorithms \cite{yang2020}.

\vspace{0.3cm}
\textbf{Acknowledgments.} \noindent The work of F.C. was carried out under the auspices of the NNSA of the U.S. DoE at LANL under Contract No. DE-AC52-06NA25396, and in particular grant PRD20190195 from the LDRD. F. C. would also like to thank W. Bruinsma for various comments and observations on the paper.

\bibliography{PEDS}

\end{document}